\documentclass[english]{article}
\usepackage[T1]{fontenc}
\usepackage[latin9]{inputenc}
\usepackage{geometry}
\usepackage{amsmath}
\usepackage{amsthm}
\usepackage{amssymb}
\usepackage{esint}

\makeatletter

\providecommand{\tabularnewline}{\\}

\theoremstyle{plain}
\newtheorem{thm}{\protect\theoremname}
\theoremstyle{plain}
\newtheorem{lem}[thm]{\protect\lemmaname}
\theoremstyle{plain}
\newtheorem{prop}[thm]{\protect\propositionname}
\theoremstyle{remark}
\newtheorem{rem}[thm]{\protect\remarkname}

\usepackage{babel}

\makeatother

\usepackage{babel}
\providecommand{\lemmaname}{Lemma}
\providecommand{\propositionname}{Proposition}
\providecommand{\remarkname}{Remark}
\providecommand{\theoremname}{Theorem}

\begin{document}
\title{Stochastic Clearing Systems with Multiple Input Processes}
\author{$\text{Bo Wei}$$\qquad$$\text{S{\i}la Çetinkaya}$$\qquad$ Daren
B.H. Cline}
\date{}
\maketitle
\begin{abstract}
In this paper, we consider stochastic clearing systems with multiple
drifted Brownian motion inputs. First, we propose an instantaneous
rate policy, which is shown to be the optimal one among a large class
of renewal type clearing policies in terms of average cost. Second,
we propose a service measure about average weighted delay rate, and
provide a unified method to calculate the service measure under different
clearing policies. Moreover, we prove that under a fixed clearing
frequency, the instantaneous rate policy outperforms a large class
of clearing policies, and the instantaneous rate hybrid policy performs
better than time-based policy, in terms of average weighted delay
rate.
\end{abstract}

\section{Introduction}

\emph{``Stochastic clearing systems are characterized by a stochastic
input process and an output mechanism that intermittently clears the
system''} \cite{Stid77}. A discussion of applications can be found
in \cite{KPS03,Stid74,Stid77,Stid86,Whit81}. In logistics, shipment
consolidation is the strategy of combining small size shipments or
customer orders, i.e., input process realizations, into a larger load.
The purpose of shipment consolidation is achieving scale economies
and increasing resource utilization. The customer orders represent
the stochastic input process. The consolidated loads are dispatched
at specific times that correspond to clearing instances. Hence, a
shipment consolidation system can be considered as a stochastic clearing
system. For practical examples of shipment consolidation, the reader
is referred to \cite{CB03}. Vehicle dispatching is another major
application of stochastic clearing systems. Passengers arrive at the
bus station randomly, and a vehicle dispatching policy determines
the capacity and instants in time at which vehicles are dispatched,
e.g. \cite{RT82,ROSS69,TZ79,ZT80}. The passengers arrival represent
as the stochastic input process, and the vehicle dispatching policy
is considered as the output clearing mechanism.

In this paper, we are interested in the cost-based optimization and
service performance of stochastic clearing model with multiple input
processes. In particular, we develop a unified method from a martingale
point of view to calculate the average cost and the average weighted
delay rate (a service measure we proposed) under different clearing
policies. Moreover, we provide several comparative results and optimization
solutions among alternative clearing policies in terms of average
cost and average weighted delay rate.

\subsection{Related Work}

In stochastic clearing systems literature, \cite{Stid74} considers
the case that the system is cleared when the quantity in the system,
$y$, exceeds the threshold $q$, and derives the explicit expression
of the limiting distribution of the quantity in the system. \cite{Stid77}
studies the optimal level of $q$, to minimize the average cost, where
there are fixed clearing and variable holding costs. In \cite{KPS03},
the stock level process is assumed as a superposition of a drifted
Brownian motion and a compound Poisson process, reflected at zero
and some cost functionals for this stochastic clearing system are
introduced under several clearing policies. However, no optimization
issues are considered in \cite{KPS03}. For the other work in stochastic
clearing systems, see \cite{Stid86,Whit81}.

As mentioned, shipment consolidation is an application of stochastic
clearing systems. In all previous work in shipment consolidation,
only specific consolidation policies have been investigated. Three
classes of shipment consolidation policies are common in practice:
quantity-based policy (QP), time-based policy (TP), and hybrid policy
(HP). The QP is aimed at consolidating a target load before releasing
a shipment to assure scale economies. Under a time-based policy, consolidated
shipments are released at periodic intervals to achieve timely delivery.
Under HP, the goal is to consolidate a target load. However, if the
time since the last shipment epoch exceeds a certain limit, then a
shipment decision is made immediately \cite{MCB10}. Early work in
shipment consolidation model focuses on simulation approaches. For
a review of earlier work, see \cite{CETI04}. More recent work places
an emphasis on analytical models. A detailed account of the analytical
literature is provided in \cite{CETI04} and \cite{MCB10}. Previous
analytical work on shipment consolidation models assumes the input
process (also referred as demand process or arrival process) is a
Poisson process \cite{CL00,CML06,MCB10}, or a renewal process \cite{CB03,CTL08},
or a discrete time Markov chain \cite{BCH11,HB95}.

Most of previous results in shipment consolidation are aimed at optimization
of alternative policies under cost-based criteria. It is worth noting
that \cite{HB94}, \cite{CML06} and \cite{CMW14} consider the service
performance of the practical shipment consolidation policies introduced
above. According to the simulation result in \cite{HB94}, QP achieves
lower average cost than TP and HP. However, in terms of average waiting
time, HP outperforms QP and TP when parameter values are fixed. Using
simulation, in the integrated inventory/shipment consolidation setting,
\cite{CML06} reveal that, although HP is not superior to QP in terms
of the cost criteria, it is superior in terms of a service measure:
average waiting time. However, the observations in \cite{HB94} and
\cite{CML06} are based on detailed simulation studies. Recently,
\cite{CMW14} attempts to provide an analytical comparison for the
maximum waiting time (MWT) and the average order delay(AOD). Specifically,
they show that under fixed policy parameters, $q$ and/or $T$, HP
outperforms QP and TP, in terms not only of $P(MWT>t)$ (for any $t>0$),
but also of AOD. On the other hand, under a fixed expected consolidation
cycle length, QP achieves the least AOD, compared with all other practical
policies.

Another application of stochastic clearing systems is found in vehicle
dispatching. The vehicle dispatching with non-stationary Poisson arrival
is studied in \cite{RT82}, and the optimality of some dispatching
policy is shown by impulsive control of jump Markov processes. Three
policies are proposed in \cite{TZ79} for vehicle dispatching, (i)
a $C$-capacity policy; (ii) a dispatching frequency policy $T$;
(iii) a $(T,C)$ policy. The average cost models are derived under
the three policies, and two firm models with cooperative and non-cooperative
solution modes are discussed.

In queueing system, different operating control policies are also
proposed in cost-based optimization models. A queueing control model
becomes a stochastic clearing system if the service rate is infinite.
\cite{YN63} introduces the concept of a controllable queueing system.
\cite{YN63} and \cite{Heym68} study the $N$-policy, where the server
restarts providing service when there are $N$ waiting customers present
in the system after the end of last busy period. \cite{Heym77} introduces
the $T$-policy, where the server is turned on after an interval of
$T$ units time, provided that the server finds any customers waiting
in the system, and shows that the optimal $N$-policy performs better
than the optimal $T$-policy in terms of average cost. \cite{Bala73}
and \cite{BT75} introduce the $D$-policy, which is to turn the server
on when the total workload for all customers in the waiting line reaches
$D$. \cite{GRS95} considers the distributions and first moments
of the busy and idle periods in controllable M/G/1 queueing systems
operating under simple and dyadic policies. Moreover, several works
are aimed at comparing those operating policies based on different
cost criteria \cite{Arta02,Boxm76,FK02,LM00}.

\subsection{Contributions}

We summarize our contributions as follows:
\begin{enumerate}
\item To the best of our knowledge, this is the first work dealing with
stochastic clearing systems with multiple input processes (drifted
Brownian motions). We point out that in this setting, the optimal
quantity based policy may not achieve less average cost than the optimal
time based policy, which is essentially different from the result
in single drifted Brownian motion input process case.
\item We identify a set of $(T_{Q}+T)$ type policies and show that the
jointly optimal $(T_{Q}+T)$-policy is either the optimal quantity
policy, or the optimal time based policy. More importantly, an instantaneous
rate policy (IRP) is proposed, which is shown to be the optimal one
among a large class of renewal type clearing policies, in terms of
average cost.
\item We provide a unified method to calculate both average cost and average
weighted delay rate (AWDR) for a class of renewal type clearing policies,
from a martingale perspective and with the aid of the martingale stopping
theorem.
\item In terms of AWDR, we show that with a fixed clearing frequency, an
IRP achieves the lowest AWDR, among a large class of renewal-type
clearing policies. Based on IRP, an instantaneous rate hybrid policy
(IRHP) (which has an upper bound on the cycle time, in contrast to
IRP) is proposed and a noteworthy result is that, with a fixed clearing
frequency, IRHP achieves less AWDR and less average cost than the
time-based policy.
\end{enumerate}
The remainder of this work is organized as follows. In Section \ref{sec:Problem-Description},
we give the problem description and define average cost and service
performance criteria. Section \ref{sec:Average-Cost-Model} provides
several comparative results and optimization solution using long-run
average cost criterion among different types of clearing policies.
In Section \ref{sec:Service-Performance-Model}, we propose measuring
performance with the average weighted delay rate, and provide the
comparative results in terms of it under a fixed clearing frequency.
Finally, the paper concludes in Section \ref{sec:Conclusions}.

\section{Problem Description\label{sec:Problem-Description}}

Assume there are $n$ different types of items, and the cumulative
demand of the $i$-th type of items $N_{i}(t)$ is a Brownian motion
with drift given by $N_{i}(t)=D_{i}t+\sigma_{i}B_{i}(t)$, where $i=1,2,\ldots,n$,
$D_{i}>0$, $\sigma_{i}>0$ are the drift coefficient and diffusion
coefficient, respectively, and $B_{1}(t),B_{2}(t),\ldots,B_{n}(t)$
are independent standard Brownian motions. Denote $D\triangleq\sum_{i=1}^{n}D_{i}$,
and $\sigma^{2}\triangleq\sum_{i=1}^{n}\sigma_{i}^{2}$. We assume
that different types of items have different unit transport cost,
and different waiting costs per unit per unit time since customers
have a distinctly different waiting sensitivity for different types
of items. All items would be packaged at the collection depot and
await the delivery. We take into account for the following parameters:
$A_{D}$ is the fixed cost of a clearing, for $i=1,2,\ldots,n$, $c_{i}$
is transport cost for one unit item of the $i$-th type, and $\omega_{i}$
is the customer waiting cost for the $i$-th type of items per unit
per unit time.

We only consider renewal-type clearing policies. Under a renewal-type
clearing policy, the consolidated load forms a regenerative process
with the clearing instants as regeneration points. This regenerative
process structure allows us to employ the renewal arguments. In this
work, we consider two criteria, one is for cost-based model, and the
other one is for service performance based model, which are introduced
as follows.

\subsection{Average Cost Criterion\label{subsec:Average-Cost-Criterion}}

The first objective of this work is analyzing and optimizing the average
cost criterion. Each renewal-type clearing policy corresponds to a
clearing cycle $\tau$, which is a stopping time. Under a clearing
policy with clearing cycle $\tau$, shipment cost within one cycle
is $A_{D}+\mathrm{\mathbb{E}}[\sum_{i=1}^{n}c_{i}N_{i}(\tau)]$, and
waiting cost within one cycle is $\mathrm{\mathbb{E}}[\sum_{i=1}^{n}\int_{0}^{\tau}\omega_{i}N_{i}(u)du]$,
and thus the long-run average cost per unit-time can be obtained by
the renewal reward theorem as follows,
\[
AC=\frac{\mathbb{E}[\text{Clearing\ \ Cycle\ \ Cost}]}{\mathbb{E}[\text{Clearing\ \ Cycle\ \ Length}]}=\frac{A_{D}+\mathrm{\mathbb{E}}[\sum_{i=1}^{n}c_{i}N_{i}(\tau)]+\mathrm{\mathbb{E}}[\sum_{i=1}^{n}\int_{0}^{\tau}\omega_{i}N_{i}(u)du]}{\mathrm{\mathbb{E}}[\tau]}.
\]

In Section \ref{sec:Average-Cost-Model}, we will propose several
clearing policies and provide the associated optimization results
in terms of average cost.

\subsection{Service Performance Criterion\label{subsec:Service-Performance-Criterion}}

The second objective of this work is to analyze and optimize a service
criterion. Customer waiting occurs in stochastic clearing systems,
since the input is not cleared immediately, and instead, the inputs
are accumulated before each clearing. One important service measure
indicator is average weighted delay per unit time before each clearing
action. This indicator is similar to AOD proposed in \cite{CMW14},
and readers who are interested in its managerial implications are
referred to \cite{CMW14}. Under a clearing policy with clearing cycle
$\tau$, the average weighted delay rate can be obtained by applying
the renewal reward theorem, i.e.,
\begin{eqnarray*}
AWDR & = & \frac{\mathbb{E}[\mbox{Cumulative weighted waiting delay per clearing cycle}]}{\mathrm{\mathbb{\mathrm{\mathbb{E}}}}[\mbox{Clearing cycle length}]}\\
 & = & \frac{\mathbb{E}[W]}{\mathbb{E}[L]}=\frac{\mathbb{E}[\sum_{i=1}^{n}\int_{0}^{\tau}\omega_{i}N_{i}(u)du]}{\mathbb{E}[\tau]}.
\end{eqnarray*}
where $W$ denotes the cumulative weighted waiting delay within one
consolidation cycle, and $L$ denotes the consolidation cycle length.
In Section \ref{sec:Service-Performance-Model}, we will propose several
clearing policies and provide comparative results in terms of AWDR.

\section{Average Cost Model\label{sec:Average-Cost-Model}}

In this section, we propose several clearing policies and discuss
the optimization results in terms of average cost criterion. Based
on a martingale argument, a unified formula is provided to compute
the expected cumulative waiting time in Subsection \ref{subsec:A-Unified-Formula}.
In Subsection \ref{subsec:Quantity Policy and Time Policy}, we calculate
the average costs under quantity policy and time policy respectively,
and point out the optimal quantity policy may not achieve less average
cost than the optimal time policy, which is essentially different
from the single input case. Further, in Subsection \ref{subsec:T_Q+T Policy},
we propose $(T_{Q}+T)$ type polices and show that either the optimal
quantity policy or the optimal time policy is the best one, depending
on whether $\sum_{i=1}^{n}\omega_{i}(2D\sigma_{i}^{2}-D_{i}\sigma^{2})$
is positive or negative. Later on, we propose an instantaneous rate
policy and demonstrate the optimal IRP achieves less average cost
than all $(T_{Q}+T)$ type polices in Subsection \ref{subsec:Instantaneous-Rate-Policy}.
Finally, in Subsection \ref{subsec:Martingale-Argument-for IRP},
we prove that the optimal IRP achieves the least average cost, among
all renewal type clearing policies with which cycle times are of finite
second moment.

\subsection{\label{subsec:A-Unified-Formula}A Unified Formula for Expected Cumulative
Waiting Time}

The goal of this subsection is to provide a unified formula to compute
the expected cumulative waiting time for $i$-th type of items within
one clearing cycle $\tau$, i.e., $\mathbb{E}\left[\int_{0}^{\tau}N_{i}(u)du\right]$,
for any $i=1,2,\ldots,n$. This formula is derived from a martingale
with the aid of the martingale stopping theorem. We denote the natural
filtration $\{\mathcal{G}_{t}\}$, which is the $\sigma$ field generated
by the family of demand process $\left\{ N_{1}(s),N_{2}(s),\ldots,N_{n}(s),s\in[0,t]\right\} $.
The following lemma reveals this martingale.
\begin{lem}
\label{martingale} For any $i=1,2,\ldots,n$,
\[
\left\{ \int_{0}^{t}N_{i}(u)du-\frac{1}{2D_{i}}N_{i}^{2}(t)+\frac{\sigma_{i}^{2}}{2D_{i}^{2}}N_{i}(t)\right\} _{t\geq0}
\]
is a martingale with respect to the natural filtration $\{\mathcal{G}_{t}\}$.
\end{lem}

\begin{proof}
Since the drifted Brownian motion $N_{i}(t)$ has stationary independent
increment, for $s<t$, we have,
\begin{eqnarray*}
\mathrm{\mathbb{E}}\Bigl[\int_{0}^{t}N_{i}(u)du\mid\mathcal{G}_{s}\Bigr] & = & \int_{0}^{s}N_{i}(u)du+\mathrm{\mathbb{E}}\Bigl[\int_{s}^{t}N_{i}(u)du\mid\mathcal{G}_{s}\Bigr]\\
 & = & \int_{0}^{s}N_{i}(u)du+(t-s)N_{i}(s)+\mathrm{\mathbb{E}}\Bigl[\int_{0}^{t-s}N_{i}(u)du\Bigr]\\
 & = & \int_{0}^{s}N_{i}(u)du+(t-s)N_{i}(s)+\frac{1}{2}D_{i}(t-s)^{2},
\end{eqnarray*}

\[
\begin{array}{rcl}
\frac{1}{2D}\mathrm{\mathbb{E}}[N_{i}^{2}(t)\mid\mathcal{G}_{s}] & = & \frac{1}{2D_{i}}\left(\mathrm{\mathbb{E}}[(N_{i}(t)-N_{i}(s))^{2}\mid\mathcal{G}_{s}]+2N_{i}(s)\mathrm{\mathbb{E}}[N_{i}(t)-N_{i}(s)\mid\mathcal{G}_{s}]+N_{i}^{2}(s)\right)\\
 & = & \frac{1}{2D_{i}}\Bigl(\sigma_{i}^{2}(t-s)^{2}+D_{i}^{2}(t-s)+2D_{i}(t-s)N_{i}(s)+N_{i}^{2}(s)\Bigr),
\end{array}
\]
and
\[
\frac{\sigma_{i}^{2}}{2D_{i}^{2}}\mathrm{\mathbb{E}}[N_{i}(t)\mid\mathcal{G}_{s}]=\frac{\sigma_{i}^{2}}{2D_{i}^{2}}(N_{i}(s)+D_{i}(t-s)).
\]
Therefore,

\[
\mathrm{\mathbb{E}}\Bigl[\int_{0}^{t}N_{i}(u)du-\frac{1}{2D_{i}}N_{i}^{2}(t)+\frac{\sigma_{i}^{2}}{2D_{i}^{2}}N_{i}(t)\mid\mathcal{G}_{s}\Bigr]=\int_{0}^{s}N_{i}(u)du-\frac{1}{2D_{i}}N_{i}^{2}(s)+\frac{\sigma_{i}^{2}}{2D_{i}^{2}}N_{i}(s),
\]
which shows that $\int_{0}^{t}N_{i}(u)du-\frac{1}{2D_{i}}N_{i}^{2}(t)+\frac{\sigma_{i}^{2}}{2D_{i}^{2}}N_{i}(t)$
is a martingale.
\end{proof}
The next result gives a unified formula to calculate the expected
cumulative waiting time for $i$-th type of items within one clearing
cycle.
\begin{prop}
\label{unifiedformula}Let $\tau$ be a stopping time with finite
second moment, i.e., $\mathrm{\mathbb{E}}[\tau^{2}]<\infty$, then
the expected cumulative waiting time for $i$-th type ( $i=1,2,\ldots,n$)
of items within one clearing cycle $\tau$ is

\[
\mathbb{E}[\int_{0}^{\tau}N_{i}(u)du]=\frac{1}{2D_{i}}\mathbb{E}[N_{i}^{2}(\tau)]-\frac{\sigma_{i}^{2}}{2D_{i}^{2}}\mathbb{E}[N(\tau)]=\frac{1}{2}D_{i}\mathrm{\mathbb{E}}\left[\tau^{2}\right]+\sigma_{i}\mathrm{\mathbb{E}}\left[\tau B_{i}(\tau)\right].
\]
\end{prop}

\begin{proof}
From Lemma \ref{martingale} and martingale convergence theorem, it
is sufficient to show that

\[
\left\{ \int_{0}^{\tau\wedge t}N_{i}(u)du-\frac{1}{2D_{i}}N_{i}^{2}(\tau\wedge t)+\frac{\sigma_{i}^{2}}{2D_{i}^{2}}N_{i}(\tau\wedge t)\right\} _{t\geq0}
\]
is uniformly integrable. In the following, we will show $\left\{ N_{i}(\tau\wedge t)\right\} _{t\geq0}$,
$\left\{ N_{i}^{2}(\tau\wedge t)\right\} _{t\geq0}$, and $\left\{ \int_{0}^{\tau\wedge t}N_{i}(u)du\right\} _{t\geq0}$
are uniformly integrable, respectively. First, for any $t\geq0$,
$\mathbb{E}[N_{i}(\tau\wedge t)]=D_{i}\mathbb{E}[\tau\wedge t]<D_{i}\mathbb{E}[\tau]<\infty$,
which implies $\left\{ N_{i}(\tau\wedge t)\right\} _{t\geq0}$ is
uniformly integrable. Second, for any $t\geq0$,

\[
\mathbb{E}[N_{i}^{2}(\tau\wedge t)]\leq2D_{i}^{2}\mathbb{E}[(\tau\wedge t)^{2}]+2\sigma_{i}^{2}\mathbb{E}[B_{i}^{2}(\tau\wedge t)]\leq D_{i}^{2}\mathbb{E}[\tau^{2}]+2\sigma_{i}^{2}\mathbb{E}[\tau]\leq\infty,
\]
which implies $\left\{ N_{i}^{2}(\tau\wedge t)\right\} _{t\geq0}$
is uniformly integrable. Third, to show $\left\{ \int_{0}^{\tau\wedge t}N_{i}(u)du\right\} _{t\geq0}$
is uniformly integrable, it is enough to show $\intop_{0}^{\infty}\mathbb{E}\left[\left|B_{i}(u)\right|1_{\tau\geq u}\right]du=\mathbb{E}\left[\intop_{0}^{\tau}\left|B_{i}(u)\right|du\right]<\infty$
from $\left|\int_{0}^{\tau\wedge t}B_{i}(u)du\right|\leq\intop_{0}^{\tau}\left|B_{i}(u)\right|du$
for all $t\geq0$. In fact, using the Hölder's inequality \cite[Theorem 3.1.11]{AL06}
with $q\in(1,4/3)$ and $p$ such that $1/p+1/q=1$, we obtain

\[
\mathbb{E}\left[\left|B_{i}(u)\right|1_{\tau\geq u}\right]\leq\left(\mathbb{E}[\left|B_{i}(u)\right|^{p}]\right)^{1/p}\left(\mathbb{E}[1_{\tau\geq u}^{q}]\right)^{1/q}\leq\left(\sigma^{p}u^{p/2}2^{p/2}\frac{\Gamma(\frac{p+1}{2})}{\sqrt{\pi}}\right)^{1/p}\left(\frac{\mathbb{E}[\tau^{2}]}{u^{2}}\right)^{1/q},
\]
where the second inequality is from $\mathbb{E}[\left|B_{i}(u)\right|^{p}]=\sigma^{p}u^{p/2}2^{p/2}\frac{\Gamma(\frac{p+1}{2})}{\sqrt{\pi}}$
in \cite{winkelbauer2012moments}, and the Markov's inequality $\mathbb{P}(\tau\geq u)\leq\frac{\mathbb{E}[\tau^{2}]}{u^{2}}$.
Thus, recall $q<3/4$, $\intop_{0}^{\infty}u^{\frac{1}{2}-\frac{2}{q}}du<\infty$,
and therefore $\intop_{0}^{\infty}\mathbb{E}\left[\left|B_{i}(u)\right|1_{\tau\geq u}\right]du<\infty$.
In sum, $\left\{ \int_{0}^{\tau\wedge t}N_{i}(u)du-\frac{1}{2D_{i}}N_{i}^{2}(\tau\wedge t)+\frac{\sigma_{i}^{2}}{2D_{i}^{2}}N_{i}(\tau\wedge t)\right\} _{t\geq0}$
is a uniformly integrable martingale, and we arrive at the conclusion.
\end{proof}
Each renewal-type clearing policy corresponds to a clearing cycle
$\tau$. From Proposition \ref{unifiedformula}, for a renewal-type
clearing policy with clearing cycle $\tau$, as long as $\mathrm{\mathbb{E}}\left[\tau^{2}\right]$
and $\mathrm{\mathbb{E}}\left[\tau B_{i}(\tau)\right]$ are obtained,
we can immediately calculate the expected cumulative waiting time
for $i$-th type of items within one clearing cycle. Moreover, the
expected total waiting cost within one clearing cycle can then be
obtained by

\[
\mathbb{E}\left[\sum_{i=1}^{n}\int_{0}^{\tau}\omega_{i}N_{i}(u)du\right]=\frac{1}{2}\sum_{i=1}^{n}\omega_{i}D_{i}\mathrm{\mathbb{E}}\left[\tau^{2}\right]+\sum_{i=1}^{n}\omega_{i}\sigma_{i}\mathrm{\mathbb{E}}\left[\tau B_{i}(\tau)\right].
\]
In the following subsections, we will propose several specific renewal-type
clearing policies, and compute the expected total waiting costs within
one clearing cycle by invoking the above formula.

\subsection{Quantity-based Policy and Time-based Policy\label{subsec:Quantity Policy and Time Policy}}

First, we adopt a quantity-based consolidation policy, which dispatches
a consolidated load when an economical dispatch quantity $Q$ is available.
Since the demands of all items are continuous, the dispatch quantity
is exactly $Q$. Define $T_{Q}=\inf\{t>0:\sum_{i=1}^{n}N_{i}(t)\geq Q\}$,
and clearly, the successive outbound shipping time intervals are independent
identically distributed, and each one has the same distribution as
the random variable $T_{Q}$. We have the following result that characterizes
the statistical property of $T_{Q}$. From \cite[Proposition 3.3]{Harr13},
we have
\begin{lem}
\label{lap-quantity} For $s>0$,
\[
\mathrm{\mathbb{E}}[\exp(-sT_{Q})]=\exp\left(-\frac{\sqrt{D^{2}+2s\sigma^{2}}-D}{\sigma^{2}}Q\right),
\]
\[
\mathrm{\mathbb{E}}[T_{Q}]=\frac{Q}{D},\quad\mathrm{\mathbb{E}}[T_{Q}^{2}]=\frac{Q^{2}}{D^{2}}+\frac{\sigma^{2}Q}{D^{3}}.
\]
In fact, $T_{Q}$ has the inverse Gaussian distribution.
\end{lem}

The next result gives joint moment generation function for $(B_{i}(T_{Q}),T_{Q})$.
\begin{lem}
\label{joint-quantity} For $s_{1}^{2}+2s_{2}<0$, and any $i=1,2,\ldots,n$,
\begin{eqnarray*}
 &  & \mathrm{\mathbb{\mathrm{\mathbb{E}}}}[\exp(s_{1}B_{i}(T_{Q})+s_{2}T_{Q})]\\
= &  & \exp\left(\frac{s_{1}\sigma_{i}+D-\sqrt{(s_{1}\sigma_{i}+D)^{2}-(s_{1}^{2}+2s_{2})\sigma^{2}}}{\sigma^{2}}Q\right),
\end{eqnarray*}
\[
\mathrm{\mathbb{E}}[B_{i}(T_{Q})T_{Q}]=-\frac{\sigma_{i}Q}{D^{2}}.
\]
\end{lem}

\begin{proof}
The proof is in the Appendix.
\end{proof}
The next result gives the expected cumulative waiting time for $i$-th
type of items within one clearing cycle under the quantity-based policy.
\begin{prop}
\label{waitingtime-quantity} Under the quantity-based policy with
parameter $Q$, the expected cumulative waiting time for $i$-th type
of items within one clearing cycle is $\frac{D_{i}Q^{2}}{2D^{2}}+\frac{D_{i}\sigma^{2}Q}{2D^{3}}-\frac{\sigma_{i}^{2}Q}{D^{2}}$.
\end{prop}

\begin{proof}
Under the quantity-based policy with parameter $Q$, from Proposition
\ref{unifiedformula}, the expected cumulative waiting time for $i$-th
type of items within one clearing cycle can be calculated by

\[
\mathrm{\mathbb{E}}\left[\int_{0}^{T_{Q}}N_{i}(t)dt\right]=\frac{1}{2}D_{i}\mathrm{\mathbb{E}}\left[T_{Q}^{2}\right]+\sigma_{i}\mathrm{\mathbb{E}}\left[T_{Q}B_{i}(T_{Q})\right].
\]
From Lemma \ref{lap-quantity} and Lemma \ref{joint-quantity}, we
obtain\\

\[
\mathrm{\mathbb{E}}\left[\int_{0}^{T_{Q}}N_{i}(t)dt\right]=\frac{D_{i}Q^{2}}{2D^{2}}+\frac{D_{i}\sigma^{2}Q}{2D^{3}}-\frac{\sigma_{i}^{2}Q}{D^{2}}.
\]
\end{proof}
\begin{rem}
\label{Poissonasaspecialcase}Suppose the demand of the $i$-type
of items $N_{i}(t)$ is a Poisson process with rate $\lambda_{i}$,
$i=1,2,\ldots,n$, the expected cumulative waiting time for the $i$-type
of items within a consolidation cycle under the quantity policy with
parameter $Q$ is
\[
\mathrm{\mathbb{E}}[\int_{0}^{T_{Q}}N_{i}(t)dt]=\mathrm{\mathbb{E}}[tN_{i}(t)|_{t=0}^{T_{Q}}]-\mathrm{\mathbb{E}}[\int_{0}^{T_{Q}}tdN_{i}(t)]=\mathrm{\mathbb{E}}[T_{Q}N_{i}(T_{Q})]-\mathrm{\mathbb{E}}[\int_{0}^{T_{Q}}tdN_{i}(t)].
\]
Clearly, the total demand $N(t)$ is a Poisson process with rate $\lambda=\sum_{i=1}^{n}\lambda_{i}$,
and $T_{Q}$ is a random variable having gamma(Q,$\lambda$) distribution,
which has mean $\frac{Q}{\lambda}$ and variance $\frac{Q}{\lambda^{2}}$.
Notice that $\mathrm{\mathbb{E}}[T_{Q}N_{i}(T_{Q})]=\mathrm{\mathbb{E}}[T_{Q}\mathrm{\mathbb{E}}[N_{i}(T_{Q})|T_{Q}]]=\frac{\lambda_{i}}{\lambda}Q\mathrm{\mathbb{E}}[T_{Q}]=\frac{\lambda_{i}Q^{2}}{\lambda^{2}}$.
Further,$\int_{0}^{t}sdN_{i}(s)-\int_{0}^{t}\lambda_{i}sds$ is a
square integrable martingale if $N_{i}(t)$ is a Poisson process,
then by the martingale stopping theorem, we have that $\mathrm{\mathbb{E}}[\int_{0}^{T_{Q}}tdN_{i}(t)]=\frac{1}{2}\lambda_{i}\mathrm{\mathbb{E}}[T_{Q}^{2}]=\frac{1}{2}\lambda_{i}(\frac{Q}{\lambda^{2}}+\frac{Q^{2}}{\lambda^{2}})$.
Therefore, the cumulative waiting time for the $i-th$ item within
a consolidation cycle is $\mathrm{\mathbb{E}}[\int_{0}^{T_{Q}}N_{i}(t)dt]=\frac{\lambda_{i}(Q-1)Q}{2\lambda^{2}}$.

By approximating the Poisson processes $N_{i}(t)$ ($i=1,2,\ldots,n$)
by drifted Brownian motion with $D_{i}=\sigma_{i}^{2}=\lambda_{i}$,
from Proposition \ref{waitingtime-quantity}, the expected cumulative
waiting time for $i$-type of items within one clearing cycle is $\frac{D_{i}Q^{2}}{2D^{2}}+\frac{D_{i}\sigma^{2}Q}{2D^{3}}-\frac{\sigma_{i}^{2}Q}{D^{2}}=\frac{\lambda_{i}(Q-1)Q}{2\lambda^{2}}$,
which is exactly the same result as above.
\end{rem}

Under the quantity-based policy with parameter $Q$, the expected
shipping cost within one clearing cycle is $A_{D}+\mathrm{\mathbb{E}}[\sum_{i=1}^{n}c_{i}N_{i}(T_{Q})]=A_{D}+\frac{Q}{D}\sum_{i=1}^{n}c_{i}D_{i}$,
the expected total waiting cost within one clearing cycle is $\sum_{i=1}^{n}\omega_{i}(\frac{D_{i}Q^{2}}{2D^{2}}+\frac{D_{i}\sigma^{2}Q}{2D^{3}}-\frac{\sigma_{i}^{2}Q}{D^{2}})$,
and the expected clearing cycle length is $\mathrm{\mathbb{E}}[T_{Q}]=\frac{Q}{D}$.
Therefore, we have the average cost by the renewal reward theorem,

\[
\begin{array}{rcl}
AC^{QP}(Q) & = & \frac{A_{D}+\frac{Q}{D}\sum_{i=1}^{n}c_{i}D_{i}+\sum_{i=1}^{n}\omega_{i}(\frac{D_{i}Q^{2}}{2D^{2}}+\frac{D_{i}\sigma^{2}Q}{2D^{3}}-\frac{\sigma_{i}^{2}Q}{D^{2}})}{Q/D}\\
 & = & \frac{A_{D}D}{Q}+\frac{Q}{2D}\sum_{i=1}^{n}\omega_{i}D_{i}+\sum_{i=1}^{n}c_{i}D_{i}-\sum_{i=1}^{n}\omega_{i}(\frac{\sigma_{i}^{2}}{D}-\frac{D_{i}\sigma^{2}}{2D^{2}}).
\end{array}
\]
We obtain the optimal quantity

\[
Q^{*}=\sqrt{\frac{2A_{D}}{\sum_{i=1}^{n}\omega_{i}D_{i}}}D
\]
and the associated average cost
\begin{equation}
AC^{QP}(Q^{*})=\sqrt{2A_{D}\sum_{i=1}^{n}\omega_{i}D_{i}}+\sum_{i=1}^{n}c_{i}D_{i}-\sum_{i=1}^{n}\omega_{i}(\frac{\sigma_{i}^{2}}{D}-\frac{D_{i}\sigma^{2}}{2D^{2}}).\label{minimizedACunderQP}
\end{equation}

Next, we adopt a time based policy, which clears the system every
$T$ units time. The expected shipping cost within one clearing cycle
is $A_{D}+\mathrm{\mathbb{E}}[\sum_{i=1}^{n}c_{i}N_{i}(T)]=A_{D}+\sum_{i=1}^{n}c_{i}D_{i}T$
and the expected total waiting cost within one clearing cycle is $\mathrm{\mathbb{E}}[\int_{0}^{T}N_{i}(t)dt]=\int_{0}^{T}D_{i}tdt=\frac{1}{2}D_{i}T^{2}$.
By the renewal reward theorem, the long-run average cost per unit-time
is
\[
AC^{TP}(T)=\frac{A_{D}+\sum_{i=1}^{n}c_{i}D_{i}T+\frac{1}{2}\sum_{i=1}^{n}\omega_{i}D_{i}T^{2}}{T}=\frac{A_{D}}{T}+\frac{1}{2}\sum_{i=1}^{n}\omega_{i}D_{i}T+\sum_{i=1}^{n}c_{i}D_{i}.
\]
We obtain the optimal time parameter $T^{*}=\sqrt{\frac{2A_{D}}{\sum_{i=1}^{n}\omega_{i}D_{i}}}$,
and the associated average cost
\begin{equation}
AC^{TP}(T^{*})=\sqrt{2A_{D}\sum_{i=1}^{n}\omega_{i}D_{i}}+\sum_{i=1}^{n}c_{i}D_{i}.\label{minimizedACunderTP}
\end{equation}

Based on (\ref{minimizedACunderQP}) and (\ref{minimizedACunderTP}),
we have the following result.
\begin{thm}
\label{QPvsTP}If $\sum_{i=1}^{n}\omega_{i}(2D\sigma_{i}^{2}-D_{i}\sigma^{2})>0$,
the optimal quantity policy achieves less average cost than the optimal
time policy; If $\sum_{i=1}^{n}\omega_{i}(2D\sigma_{i}^{2}-D_{i}\sigma^{2})<0$,
the optimal time policy achieves less average cost than the optimal
quantity policy.
\end{thm}

\begin{rem}
Mutlu et al.(2010) shows that in the single item Poisson demand case,
the optimal quantity policy achieves the lowest average cost. Theorem
\ref{QPvsTP} points out it may not be true that the optimal quantity
policy always achieves less average cost than the optimal time policy,
in a stochastic clearing system with multiple drifted Brownian motion
inputs.

(i) We consider the single input process case, i.e., $n=1$, then
$D=D_{1},$ $\sigma^{2}=\sigma_{1}^{2}$ and $\sum_{i=1}^{n}\omega_{i}(2D\sigma_{i}^{2}-D_{i}\sigma^{2})>0$.
From Theorem \ref{QPvsTP}, the optimal quantity policy always achieves
less average cost than the optimal time policy.

(ii) We consider the case with multiple independent Poisson processes
with rates $\lambda_{i}$ ($i=1,2,\ldots,n$). From Remark \ref{Poissonasaspecialcase},
we obtain exact cost term using $D_{i}=\sigma_{i}^{2}=\lambda_{i}$
in Brownian model. Notice that $\sum_{i=1}^{n}\omega_{i}(2D\sigma_{i}^{2}-D_{i}\sigma^{2})=D\sum_{i=1}^{n}\omega_{i}D_{i}>0$,
which implies that in a stochastic clearing system with multiple Poisson
inputs, the optimal quantity policy always achieves less average cost
than the optimal time policy.
\end{rem}

\subsection{$(T_{Q}+T)$-Policy\label{subsec:T_Q+T Policy}}

Given a quantity-based consolidation policy with parameter $Q$, we
consider a modified policy, denoted as $(T_{Q}+T)$-policy which dispatches
the consolidated load at a nonnegative time $T$ later than it takes
to accumulate $Q$. This type of policy also appears in \cite{KPS03,LM00}.
Obviously, quantity policy and time policy can be treated as two special
types of policies in the set of $(T_{Q}+T)$ type policies. The goal
of this subsection is to show that within the set of $(T_{Q}+T)$
type policies, either the optimal quantity policy or the optimal time
policy is optimal, depending on whether $\sum_{i=1}^{n}\omega_{i}(2D\sigma_{i}^{2}-D_{i}\sigma^{2})$
is positive or negative. We obtain the expected cumulative waiting
time for $i$-type of items within one clearing cycle under the $(T_{Q}+T)$-policy.
\begin{prop}
Under the $(T_{Q}+T)$-policy, the expected cumulative waiting time
for $i$-type of items within one clearing cycle is $\frac{D_{i}Q^{2}}{2D^{2}}+\frac{D_{i}\sigma^{2}Q}{2D^{3}}-\frac{\sigma_{i}^{2}Q}{D^{2}}+\frac{D_{i}Q}{D}T+\frac{1}{2}D_{i}T^{2}$.
\end{prop}

\begin{proof}
Under the $(T_{Q}+T)$-policy, from Proposition \ref{unifiedformula},
the expected cumulative waiting time for $i$-th type of items within
one clearing cycle can be calculated by

\[
\mathrm{\mathbb{E}}\left[\int_{0}^{T_{Q}+T}N_{i}(t)dt\right]=\frac{1}{2}D_{i}\mathrm{\mathbb{E}}\left[(T_{Q}+T)^{2}\right]+\sigma_{i}\mathrm{\mathbb{E}}\left[(T_{Q}+T)B_{i}(T_{Q}+T)\right]=\frac{1}{2}D_{i}\mathrm{\mathbb{E}}\left[(T_{Q}+T)^{2}\right]+\sigma_{i}\mathrm{\mathbb{E}}\left[T_{Q}B_{i}(T_{Q})\right].
\]
From Lemma \ref{lap-quantity} and Lemma \ref{joint-quantity}, we
obtain\\

\[
\mathrm{\mathbb{E}}\left[\int_{0}^{T_{Q}+T}N_{i}(t)dt\right]=\frac{D_{i}Q^{2}}{2D^{2}}+\frac{D_{i}\sigma^{2}Q}{2D^{3}}-\frac{\sigma_{i}^{2}Q}{D^{2}}+\frac{D_{i}Q}{D}T+\frac{1}{2}D_{i}T^{2}.
\]
\end{proof}
Under the $(T_{Q}+T)$-policy with parameters $Q$ and $T$, the expected
shipping cost within one clearing cycle is $A_{D}+(\frac{Q}{D}+T)\sum_{i=1}^{n}c_{i}D_{i}$,
the expected total waiting cost within one clearing cycle is $\sum_{i=1}^{n}\omega_{i}(\frac{D_{i}Q^{2}}{2D^{2}}+\frac{D_{i}\sigma^{2}Q}{2D^{3}}-\frac{\sigma_{i}^{2}Q}{D^{2}}+\frac{D_{i}Q}{D}T+\frac{1}{2}D_{i}T^{2})$,
and the expected clearing cycle length is $\mathbb{\mathrm{\mathbb{E}}}[T_{Q}+T]=\frac{Q}{D}+T$.
Therefore, we have the average cost by the renewal reward theorem,

\[
AC^{QTP}(Q,T)=\frac{Q+DT}{2D}\sum_{i=1}^{n}\omega_{i}D_{i}+\sum_{i=1}^{n}c_{i}D_{i}+\frac{A_{D}D-\sum_{i=1}^{n}\omega_{i}(\frac{\sigma_{i}^{2}}{D}-\frac{D_{i}\sigma^{2}}{2D^{2}})Q}{Q+DT}.
\]

The main goal of this subsection is to obtain the jointly optimal
$(T_{Q}+T)$-policy (over $Q$ and $T$). We need the following lemmas
before we achieve this goal. The following result provides that for
a fixed value $Q$, what is the condition for that some $(T_{Q}+T)$-policy
may achieve less average cost than the quantity policy with parameter
$Q$. Define

\[
\overline{Q}\triangleq\frac{-\sum_{i=1}^{n}\omega_{i}(2\sigma_{i}^{2}-\frac{D_{i}\sigma^{2}}{D})+\sqrt{[\sum_{i=1}^{n}\omega_{i}(2\sigma_{i}^{2}-\frac{D_{i}\sigma^{2}}{D}]^{2})+8A_{D}D^{2}\sum_{i=1}^{n}\omega_{i}D_{i}}}{2\sum_{i=1}^{n}\omega_{i}D_{i}},
\]
and we show that $\overline{Q}$ is the threshold value in terms of
whether a quantity policy with parameter $Q$ can be improved by a
$(T_{Q}+T)$-policy.
\begin{lem}
\label{lemmabarQ}The quantity policy with parameter $Q$ can be improved
by a $(T_{Q}+T)$-policy if and only if $Q$ satisfies the following:
\begin{eqnarray}
\sum_{i=1}^{n}\omega_{i}D_{i}Q^{2}+\sum_{i=1}^{n}\omega_{i}(2\sigma_{i}^{2}-\frac{D_{i}\sigma^{2}}{D})Q-2A_{D}D^{2}\leq0\label{barQ}
\end{eqnarray}
which is equivalent to $0\leq Q\leq\overline{Q}$.
\end{lem}

\begin{proof}
If the quantity policy with parameter $Q$ can be improved by a $(T_{Q}+T)$-policy,
that is $AC^{QP}(Q)\geq AC^{QTP}(Q,T)$,
\[
\frac{A_{D}D}{Q}+\frac{Q}{2D}\sum_{i=1}^{n}\omega_{i}D_{i}+\sum_{i=1}^{n}c_{i}D_{i}-\sum_{i=1}^{n}\omega_{i}(\frac{\sigma_{i}^{2}}{D}-\frac{D_{i}\sigma^{2}}{2D^{2}})
\]
\[
\geq\frac{Q+DT}{2D}\sum_{i=1}^{n}\omega_{i}D_{i}+\sum_{i=1}^{n}c_{i}D_{i}+\frac{A_{D}D-\sum_{i=1}^{n}\omega_{i}(\frac{\sigma_{i}^{2}}{D}-\frac{D_{i}\sigma^{2}}{2D^{2}})Q}{Q+DT}.
\]
After some algebraic manipulation, we arrive at
\[
DQ\sum_{i=1}^{n}\omega_{i}D_{i}T\leq2A_{D}D^{2}-\sum_{i=1}^{n}\omega_{i}(2\sigma_{i}^{2}-\frac{D_{i}\sigma^{2}}{D})Q-\sum_{i=1}^{n}\omega_{i}D_{i}Q^{2}.
\]
The quantity policy with parameter $Q$ can be improved by some $(T_{Q}+T)$-policy
if and only if we can choose a non-negative value of $T$ such that
the above inequality is satisfied. This is always possible if
\[
\sum_{i=1}^{n}\omega_{i}D_{i}Q^{2}+\sum_{i=1}^{n}\omega_{i}(2\sigma_{i}^{2}-\frac{D_{i}\sigma^{2}}{D})Q-2A_{D}D^{2}\leq0.
\]
Since $Q\geq0$, we only consider the positive root of the quadratic
equation. Thus, the $Q$-policy can be improved if and only if $0\leq Q\leq\overline{Q}$,
and the proof is completed.
\end{proof}
The next result provides the optimal value $T$ of the $(T_{Q}+T)$-policies,
if the parameter $Q$ satisfies (\ref{barQ}).
\begin{lem}
\label{Topt(Q)}If $Q$ satisfies (\ref{barQ}) , the optimal $(T_{Q}+T)$-policy
(over $T$) is $(Q+T^{opt})$, where $T^{opt}(Q)=\sqrt{\frac{2A_{D}-\sum_{i=1}^{n}\omega_{i}(2\sigma_{i}^{2}Q/D^{2}-D_{i}\sigma^{2}Q/D^{3})}{\sum_{i=1}^{n}\omega_{i}D_{i}}}-\frac{Q}{D}$.
\end{lem}

\begin{proof}
To optimize $AC^{QTP}(Q,T)$ as a function of $T$, we have to solve
the equation $dAC^{QTP}(Q,T)/dT=0$, which is
\[
\frac{1}{2}\sum_{i=1}^{n}\omega_{i}D_{i}-\frac{A_{D}D-\sum_{i=1}^{n}\omega_{i}(\sigma_{i}^{2}Q/D-\frac{1}{2}D_{i}\sigma^{2}Q/D^{2})}{(Q+D\tau_{1})^{2}}D=0.
\]
Hence $T^{opt}(Q)=\sqrt{\frac{2A_{D}-\sum_{i=1}^{n}\omega_{i}(2\sigma_{i}^{2}Q/D^{2}-D_{i}\sigma^{2}Q/D^{3})}{\sum_{i=1}^{n}\omega_{i}D_{i}}}-\frac{Q}{D}$.
Since $Q$ satisfies (\ref{barQ}), we can see $T^{opt}(Q)\geq0$.
\end{proof}
The following result characterizes the optimality of $(T_{Q}+T)$-policy,
optimizing jointly on $Q$ and $T$. It is somewhat surprising, which
states that either a quantity policy or a time policy is optimal,
depending on whether $\sum_{i=1}^{n}\omega_{i}(2D\sigma_{i}^{2}-D_{i}\sigma^{2})$
is positive or negative.
\begin{thm}
\label{jointoptimality} If $\sum_{i=1}^{n}\omega_{i}(2D\sigma_{i}^{2}-D_{i}\sigma^{2})>0$,
the jointly optimal $(T_{Q}+T)$-policy is the optimal quantity policy;
If $\sum_{i=1}^{n}\omega_{i}(2D\sigma_{i}^{2}-D_{i}\sigma^{2})<0$,
the joint optimal $(T_{Q}+T)$-policy is the optimal time policy.
\end{thm}

\begin{proof}
We first fix the parameter $Q$. From Lemma \ref{lemmabarQ}, we know
that if (\ref{barQ}) does not hold, i.e $Q>\overline{Q}$, the quantity
policy with $Q$ cannot be improved by a $(T_{Q}+T)$-policy. We now
focus on the values of $Q$ that satisfy (\ref{barQ}). For such $Q$,
we have the optimal value $T^{opt}(Q)$ of the $(T_{Q}+T)$-policies
from Lemma \ref{Topt(Q)}, which results in the optimal average cost
determined by the pair $(Q,T^{opt}(Q))$ as follows:
\begin{eqnarray}
AC^{QTP}(Q,T^{opt}(Q))=\sqrt{[2A_{D}-\sum_{i=1}^{n}\omega_{i}(2\sigma_{i}^{2}/D^{2}-D_{i}\sigma^{2}/D^{3})Q](\sum_{i=1}^{n}\omega_{i}D_{i})}+\sum_{i=1}^{n}c_{i}D_{i}.\label{ACQTP-jointly}
\end{eqnarray}

Now we vary $Q$ to obtain the joint optimality of the $(T_{Q}+T)$
type policies.

Case 1: $\sum_{i=1}^{n}\omega_{i}(2D\sigma_{i}^{2}-D_{i}\sigma^{2})>0$.
In this case, (\ref{ACQTP-jointly}) is a decreasing function of $Q$.
Since $0\leq Q\leq\overline{Q}$, take $Q$ as close from the left
side to $\overline{Q}$ as possible to minimize (\ref{ACQTP-jointly}).
Further, recall that and no improvement of a $(T_{Q}+T)$-policy over
the quantity policy with $Q>\overline{Q}$, therefore, the optimal
$(T_{Q}+T)$-policy is the optimal quantity policy.

Case 2: $\sum_{i=1}^{n}\omega_{i}(2D\sigma_{i}^{2}-D_{i}\sigma^{2})<0$.
In this case, (\ref{ACQTP-jointly}) is an increasing function of
$Q$. Since $0\leq Q\leq\overline{Q}$, take $Q=0$ to minimize (\ref{ACQTP-jointly}),
which implies that the optimal $(T_{Q}+T)$-policy is a time-based
policy.
\end{proof}
\begin{rem}
In Theorem \ref{QPvsTP}, we show that if $\sum_{i=1}^{n}\omega_{i}(2D\sigma_{i}^{2}-D_{i}\sigma^{2})>0$,
the optimal quantity policy achieves less average cost than the optimal
time policy; If $\sum_{i=1}^{n}\omega_{i}(2D\sigma_{i}^{2}-D_{i}\sigma^{2})<0$,
the optimal time policy achieves less average cost than the optimal
quantity policy. In Theorem \ref{jointoptimality}, we obtain a stronger
result, which claims within the set of $(T_{Q}+T)$ type policies,
the jointly optimal policy can only be either a quantity policy or
a time policy, depending on $\sum_{i=1}^{n}\omega_{i}(2D\sigma_{i}^{2}-D_{i}\sigma^{2})$
is positive or negative. Our next goal is to seek some policy which
beats all quantity policies and all time policies in terms of average
cost criterion.
\end{rem}

\subsection{Instantaneous Rate Policy\label{subsec:Instantaneous-Rate-Policy}}

We propose a new policy, where a clearing is triggered whenever the
instantaneous waiting penalty rate hits a threshold value, i.e., a
clearing is made as long as $\sum_{i=1}^{n}\omega_{i}N_{i}(t)=M$,
$M$ is a threshold value we need to optimize. We call this new policy
as an instantaneous rate policy (IRP). Recalling that under a quantity-based
policy with parameter $Q$, we clear the system as long as the total
consolidated load reaches $Q$, and under a time-based policy with
parameter $T$, the system is cleared every $T$ units time. Clearly,
under a quantity-based policy, we just need to track the total input
process as a whole. Under a time-based policy, we do not need to track
any process at all. In contrast, we need to track each input process
associated with each type, when we implement an instantaneous rate
policy.

The motivation of the instantaneous rate policy is as follows: suppose
the inputs are discrete and arrive one by one, if the first arriving
item has large waiting sensitivity, it is not economical to hold the
consolidated load for a long time; while if the first arriving item
has small waiting sensitivity, we can prolong the holding time of
the consolidated load. Upon this observation, we should realize that
the optimal policy requires tracking each input process associated
with each type.

Define $\tau_{M}=\inf\{t>0:\sum_{i=1}^{n}\omega_{i}N_{i}(t)\geq M\}$,
which is a stopping time w.r.t the filtration generated by $\left\{ B_{1}(t),B_{2}(t),\ldots,B_{n}(t)\right\} _{t\geq0}$.
From Proposition 3.3 of \cite{Harr13}, we have
\begin{lem}
\label{lap-weighted} For $s>0$,
\[
\mathrm{\mathbb{E}}[\exp(-s\tau_{M})]=\exp\left(\frac{\sum_{i=1}^{n}\omega_{i}D_{i}-\sqrt{(\sum_{i=1}^{n}\omega_{i}D_{i})^{2}+2s\sum_{i=1}^{n}\omega_{i}^{2}\sigma_{i}^{2}}}{\sum_{i=1}^{n}\omega_{i}^{2}\sigma_{i}^{2}}M\right),
\]
\[
\mathrm{\mathbb{E}}[\tau_{M}]=\frac{M}{\sum_{i=1}^{n}\omega_{i}D_{i}},\quad\mathrm{\mathbb{E}}[\tau_{M}^{2}]=\frac{M^{2}}{(\sum_{i=1}^{n}\omega_{i}D_{i})^{2}}+\frac{\sum_{i=1}^{n}\omega_{i}^{2}\sigma_{i}^{2}M}{(\sum_{i=1}^{n}\omega_{i}D_{i})^{3}}.
\]
\end{lem}

The next result gives joint moment generation function for $(B_{i}(\tau_{M}),\tau_{M})$.
\begin{lem}
\label{joint-weighted} For $s_{1}^{2}+2s_{2}<0$,
\begin{eqnarray*}
 &  & \mathrm{\mathbb{E}}[\exp(s_{1}B_{i}(\tau_{M})+s_{2}\tau_{M})]\\
= &  & \exp\left(\frac{s_{1}\omega_{i}\sigma_{i}+\sum_{i=1}^{n}\omega_{i}D_{i}-\sqrt{(s_{1}\omega_{i}\sigma_{i}+\sum_{i=1}^{n}\omega_{i}D_{i})^{2}-(s_{1}^{2}+2s_{2})\sum_{i=1}^{n}\omega_{i}^{2}\sigma_{i}^{2}}}{\sum_{i=1}^{n}\omega_{i}^{2}\sigma_{i}^{2}}M\right),
\end{eqnarray*}
\[
\mathrm{\mathbb{E}}[B_{i}(\tau_{M})\tau_{M}]=-\frac{\omega_{i}\sigma_{i}M}{(\sum_{i=1}^{n}\omega_{i}D_{i})^{2}}.
\]
\end{lem}

\begin{proof}
The proof is similar to the proof of Lemma \ref{joint-quantity}.
\end{proof}
By using Lemma \ref{lap-weighted} and Lemma \ref{joint-weighted},
we can obtain the following result which provides the expected waiting
time for the $i$-th type of items and the total waiting cost for
all types of items within one clearing cycle.
\begin{prop}
Under the instantaneous rate policy with parameter $M$, the expected
cumulative waiting time for $i$-th type of items within one clearing
cycle is
\begin{eqnarray*}
 &  & \mathrm{\mathbb{E}}\left[\int_{0}^{\tau_{M}}N_{i}(t)dt\right]=\frac{1}{2}\frac{D_{i}}{(\sum_{i=1}^{n}\omega_{i}D_{i})^{2}}M^{2}+\frac{1}{2}\frac{D_{i}\sum_{i=1}^{n}\omega_{i}^{2}\sigma_{i}^{2}}{(\sum_{i=1}^{n}\omega_{i}D_{i})^{3}}M-\frac{\omega_{i}\sigma_{i}^{2}}{(\sum_{i=1}^{n}\omega_{i}D_{i})^{2}}M,
\end{eqnarray*}
and the expected total waiting cost for all items within one dispatch
cycle is
\[
\sum_{i=1}^{n}\omega_{i}\mathrm{\mathbb{E}}\left[\int_{0}^{\tau_{M}}N_{i}(t)dt\right]=\frac{1}{2\sum_{i=1}^{n}\omega_{i}D_{i}}M^{2}-\frac{\sum_{i=1}^{n}\omega_{i}^{2}\sigma_{i}^{2}}{2(\sum_{i=1}^{n}\omega_{i}D_{i})^{2}}M.
\]
\end{prop}

\begin{proof}
From Proposition \ref{unifiedformula}, under the instantaneous rate
policy with parameter $M$, the expected cumulative waiting time for
$i$-th type of items within one clearing cycle can be calculated
by

\[
\mathrm{\mathbb{E}}\left[\int_{0}^{\tau_{M}}N_{i}(t)dt\right]=\frac{1}{2}D_{i}\mathrm{\mathbb{E}}\left[\tau_{M}^{2}\right]+\sigma_{i}\mathrm{\mathbb{E}}\left[\tau_{M}B_{i}(\tau_{M})\right].
\]
From Lemma \ref{lap-weighted} and Lemma \ref{joint-weighted}, we
obtain\\

\[
\mathrm{\mathbb{E}}\left[\int_{0}^{\tau_{M}}N_{i}(t)dt\right]=\frac{1}{2}\frac{D_{i}}{(\sum_{i=1}^{n}\omega_{i}D_{i})^{2}}M^{2}+\frac{1}{2}\frac{D_{i}\sum_{i=1}^{n}\omega_{i}^{2}\sigma_{i}^{2}}{(\sum_{i=1}^{n}\omega_{i}D_{i})^{3}}M-\frac{\omega_{i}\sigma_{i}^{2}}{(\sum_{i=1}^{n}\omega_{i}D_{i})^{2}}M.
\]
\end{proof}
Under the instantaneous rate policy with parameter $M$, the expected
shipping cost within each cycle is $A_{D}+\mathrm{\mathbb{E}}[\sum_{i=1}^{n}c_{i}N_{i}(\tau_{M})]=A_{D}+\frac{M}{\sum_{i=1}^{n}\omega_{i}D_{i}}\sum_{i=1}^{n}c_{i}D_{i}$,
the expected total waiting cost within one cycle is $\frac{1}{2\sum_{i=1}^{n}\omega_{i}D_{i}}M^{2}-\frac{\sum_{i=1}^{n}\omega_{i}^{2}\sigma_{i}^{2}}{2(\sum_{i=1}^{n}\omega_{i}D_{i})^{2}}M$,
and the expected clearing cycle length is $\mathrm{\mathbb{E}}[\tau_{M}]=\frac{M}{\sum_{i=1}^{n}\omega_{i}D_{i}}$.
Therefore, by the renewal reward theorem, we can obtain the long-run
average cost
\begin{eqnarray*}
AC^{IRP}(M) & = & \frac{A_{D}+\frac{M}{\sum_{i=1}^{n}\omega_{i}D_{i}}\sum_{i=1}^{n}c_{i}D_{i}+\frac{1}{2\sum_{i=1}^{n}\omega_{i}D_{i}}M^{2}-\frac{\sum_{i=1}^{n}\omega_{i}^{2}\sigma_{i}^{2}}{2(\sum_{i=1}^{n}\omega_{i}D_{i})^{2}}M}{\frac{M}{\sum_{i=1}^{n}\omega_{i}D_{i}}}\\
 & = & \frac{A_{D}\sum_{i=1}^{n}\omega_{i}D_{i}}{M}+\frac{1}{2}M+\sum_{i=1}^{n}c_{i}D_{i}-\frac{\sum_{i=1}^{n}\omega_{i}^{2}\sigma_{i}^{2}}{2\sum_{i=1}^{n}\omega_{i}D_{i}}.
\end{eqnarray*}
Minimizing $AC^{IRP}(M)$, we get the optimal threshold value $M^{*}=\sqrt{2A_{D}\sum_{i=1}^{n}\omega_{i}D_{i}},$
and the minimized average cost under the instantaneous rate policy
\begin{equation}
AC^{IRP}(M^{*})=\sqrt{2A_{D}\sum_{i=1}^{n}\omega_{i}D_{i}}+\sum_{i=1}^{n}c_{i}D_{i}-\frac{\sum_{i=1}^{n}\omega_{i}^{2}\sigma_{i}^{2}}{2\sum_{i=1}^{n}\omega_{i}D_{i}}.\label{minimizedACunderIRP}
\end{equation}

\begin{table}[h]
\centering{}\hspace*{-0.5in} %
\begin{tabular}{llll}
\hline
\vspace{-0.1in}
  &  &  & \tabularnewline
$QP$  & $:$  & {\Large{}{}$\;\sqrt{2A_{D}\sum_{i=1}^{n}\omega_{i}D_{i}}+\sum_{i=1}^{n}c_{i}D_{i}-\sum_{i=1}^{n}\omega_{i}(\frac{\sigma_{i}^{2}}{D}-\frac{D_{i}\sigma^{2}}{2D^{2}})$}  & \vspace{0.05in}
 \tabularnewline
$TP$  & $:$  & {\Large{}{}$\;\sqrt{2A_{D}\sum_{i=1}^{n}\omega_{i}D_{i}}+\sum_{i=1}^{n}c_{i}D_{i}$}  & \vspace{0.05in}
 \tabularnewline
$IRP$  & $:$  & {\Large{}{}$\;\sqrt{2A_{D}\sum_{i=1}^{n}\omega_{i}D_{i}}+\sum_{i=1}^{n}c_{i}D_{i}-\frac{\sum_{i=1}^{n}\omega_{i}^{2}\sigma_{i}^{2}}{2\sum_{i=1}^{n}\omega_{i}D_{i}}$}  & \vspace{0.1in}
 \tabularnewline
\hline
\end{tabular}\caption{Summary of the minimized average cost under different policies.}
\label{minAC}
\end{table}

Recall the minimized average cost under the optimal time policy (\ref{minimizedACunderTP}),
we have $AC^{TP}(T^{*})>AC^{IRP}(M^{*})$. Further, recall the minimized
average cost under the optimal quantity policy (\ref{minimizedACunderQP}),
we have
\begin{eqnarray*}
AC^{QP}(Q^{*})-AC^{IRP}(M^{*})=\frac{\sum_{k=1}^{n}\sigma_{k}^{2}\left[\omega_{k}D-\sum_{i=1}^{n}\omega_{i}D_{i}\right]^{2}}{2D^{2}\sum_{i=1}^{n}\omega_{i}D_{i}}\geq0.
\end{eqnarray*}
Therefore, the optimal IRP achieves lower average cost than both of
the optimal quantity policy and the optimal time policy. Combined
with Theorem \ref{jointoptimality}, we have the following result.
\begin{thm}
The optimal IRP achieves less average cost than all $(T_{Q}+T)$ type
policies.
\end{thm}

\begin{rem}
In a stochastic clearing system with multiple input processes, under
a time policy, we do not need to track any process realization; under
a quantity policy, we only need to track the total input processes
as a whole; under an instantaneous rate policy, we need to track the
realization of each input process. In a stochastic dynamic system,
the optimal policy should be the one taking advantage of full information,
i.e, a closed-loop policy.
\end{rem}

\subsection{Optimality of IRP\label{subsec:Martingale-Argument-for IRP}}

In this subsection, we show that among a large class of renewal type
clearing policies, the optimal IRP achieves the least average cost.
The argument is based on the following result.
\begin{prop}
\label{keyequality}Let $\tau$ be a stopping time with finite second
moment, i.e., $\mathrm{\mathbb{E}}[\tau^{2}]<\infty$, it holds that
\end{prop}

\begin{eqnarray}
 &  & \mathbb{E}\left[\sum_{i=1}^{n}\int_{0}^{\tau}\omega_{i}N_{i}(u)du\right]=\frac{1}{2\sum_{i=1}^{n}\omega_{i}D_{i}}\mathrm{\mathbb{E}}[(\sum_{i=1}^{n}\omega_{i}N_{i}(\tau))^{2}]-\frac{\sum_{i=1}^{n}\omega_{i}^{2}\sigma_{i}^{2}}{2(\sum_{i=1}^{n}\omega_{i}D_{i})^{2}}\mathrm{\mathbb{E}}\left[\sum_{i=1}^{n}\omega_{i}N_{i}(\tau)\right].\label{CWD}
\end{eqnarray}

\begin{proof}
From Proposition \ref{unifiedformula}, we have $\mathbb{E}\left[\sum_{i=1}^{n}\int_{0}^{\tau}\omega_{i}N_{i}(u)du\right]=\frac{1}{2}\sum_{i=1}^{n}\omega_{i}D_{i}\mathrm{\mathbb{E}}\left[\tau^{2}\right]+\sum_{i=1}^{n}\omega_{i}\sigma_{i}\mathrm{\mathbb{E}}\left[\tau B_{i}(\tau)\right]$.
By a direct calculation, we have

\[
\mathrm{\mathbb{E}}[(\sum_{i=1}^{n}\omega_{i}N_{i}(\tau))^{2}]=(\sum_{i=1}^{n}\omega_{i}D_{i})^{2}\mathbb{E}\left[\tau^{2}\right]+2\sum_{i=1}^{n}\omega_{i}D_{i}\sum_{i=1}^{n}\omega_{i}\sigma_{i}\mathrm{\mathbb{E}}\left[\tau B_{i}(\tau)\right]+\sum_{i=1}^{n}\omega_{i}^{2}\sigma_{i}^{2}\mathbb{E}\left[\tau\right].
\]
Based on the above two observations, and $\mathrm{\mathbb{E}}\left[\sum_{i=1}^{n}\omega_{i}N_{i}(\tau)\right]=\sum_{i=1}^{n}\omega_{i}D_{i}\mathbb{E}\left[\tau\right]$,
we arrive at the conclusion.
\end{proof}
We provide the optimality of IRP in terms of average cost as follows.
\begin{thm}
\label{IRPoptimalinAC} Among all renewal type clearing policies with
which cycle times are of finite second moment, the optimal policy
in terms of average cost is the optimal IRP.
\end{thm}

\begin{proof}
For a renewal type clearing policy with stopping time $\tau$ and
$\mathrm{\mathbb{E}}[\tau^{2}]<\infty$, using Proposition \ref{keyequality}
we have the average cost in the long run

\begin{eqnarray*}
 &  & \frac{A_{D}+\sum_{i=1}^{n}c_{i}\mathrm{\mathbb{E}}[N_{i}(\tau)]+\mathrm{\mathbb{E}}[\sum_{i=1}^{n}\int_{0}^{\tau}\omega_{i}N_{i}(u)du]}{\mathrm{\mathbb{E}}[\tau]}\\
= &  & \frac{A_{D}+\sum_{i=1}^{n}c_{i}\mathrm{\mathbb{E}}[N_{i}(\tau)]+\frac{1}{2\sum_{i=1}^{n}\omega_{i}D_{i}}\mathrm{\mathbb{E}}[(\sum_{i=1}^{n}\omega_{i}N_{i}(\tau))^{2}]-\frac{\sum_{i=1}^{n}\omega_{i}^{2}\sigma_{i}^{2}}{2(\sum_{i=1}^{n}\omega_{i}D_{i})^{2}}\mathrm{\mathbb{E}}[\sum_{i=1}^{n}\omega_{i}N_{i}(\tau)]}{\mathrm{\mathbb{E}}[\tau]}\\
\geq &  & \frac{A_{D}+\sum_{i=1}^{n}c_{i}\mathrm{\mathbb{E}}[N_{i}(\tau)]+\frac{1}{2\sum_{i=1}^{n}\omega_{i}D_{i}}\mathrm{\mathbb{E}}^{2}[\sum_{i=1}^{n}\omega_{i}N_{i}(\tau)]-\frac{\sum_{i=1}^{n}\omega_{i}^{2}\sigma_{i}^{2}}{2(\sum_{i=1}^{n}\omega_{i}D_{i})^{2}}\mathrm{\mathbb{E}}[\sum_{i=1}^{n}\omega_{i}N_{i}(\tau)]}{\mathrm{\mathbb{E}}[\tau]}
\end{eqnarray*}
where the equality follows from (\ref{CWD}), and the last inequality
comes from $\mathrm{\mathbb{E}}[(\sum_{i=1}^{n}\omega_{i}N_{i}(\tau))^{2}]\geq\mathrm{\mathbb{E}}^{2}[\sum_{i=1}^{n}\omega_{i}N_{i}(\tau)]$,
and the equality in the last inequality holds if and only if $\sum_{i=1}^{n}\omega_{i}N_{i}(\tau)$
is a constant a.s. Also, we notice that, if we fix $\mathrm{\mathbb{E}}[\tau]$,
the numerator of last term in the formula is also fixed. Therefore,
we arrive at our conclusion.
\end{proof}
\begin{rem}
Theorem \ref{IRPoptimalinAC} is a culmination in this section. At
the first stage, we show that among all $(T_{Q}+T)$ type polices,
either the optimal quantity policy or the optimal time policy is the
best one in terms of average cost, depending on whether $\sum_{i=1}^{n}\omega_{i}(2D\sigma_{i}^{2}-D_{i}\sigma^{2})$
is positive or negative. Later on, we demonstrate that the optimal
IRP achieves less average cost than all $(T_{Q}+T)$ type polices.
Finally, we prove the optimal IRP achieves the least average cost,
among all renewal type clearing policies with which cycle times are
of finite second moment.
\end{rem}

\section{Service Performance Model\label{sec:Service-Performance-Model}}

The second objective of this work is to analyze and optimize a service
criterion. We propose measuring performance with the average weighted
delay rate. Recall its definition in Subsection \ref{subsec:Service-Performance-Criterion},
under a renewal type clearing policy with clearing cycle $\tau$,
\begin{eqnarray*}
AWDR=\frac{\mathrm{\mathbb{E}}[W]}{\mathrm{\mathbb{E}}[L]}=\frac{\mathrm{\mathbb{E}}[\sum_{i=1}^{n}\int_{0}^{\tau}\omega_{i}N_{i}(u)du]}{\mathrm{\mathbb{E}}[\tau]}
\end{eqnarray*}
We index $AWDR$, $W$, and $L$ by policy type as needed. Recalling
(\ref{CWD}), we have
\begin{eqnarray*}
AWDR_{\tau}=\frac{\frac{1}{2\sum_{i=1}^{n}\omega_{i}D_{i}}\mathrm{\mathbb{E}}[(\sum_{i=1}^{n}\omega_{i}N_{i}(\tau))^{2}]-\frac{\sum_{i=1}^{n}\omega_{i}^{2}\sigma_{i}^{2}}{2(\sum_{i=1}^{n}\omega_{i}D_{i})^{2}}\mathrm{\mathbb{E}}[\sum_{i=1}^{n}\omega_{i}N_{i}(\tau)]}{\mathrm{\mathbb{E}}[\tau]},
\end{eqnarray*}
which provides a unified formula to calculate the average weighted
delay rate under any renewal-type clearing policy with which cycle
time is of finite second moment.

From the above discussion, we can deduce AWDR for any renewal-type
clearing policy. We focus on instantaneous rate policy (IRP), time
based policy (TP), and instantaneous rate hybrid policy (IRHP). Instantaneous
rate hybrid policy is a combination of IRP and TP. Stated formally,
under IRHP with parameter $M$ and $T$, the goal is to implement
an instantaneous rate policy with parameter $M$. However, if until
time $T$ since the last shipment epoch, $\sum_{i=1}^{n}\omega_{i}N_{i}(t)$
has not reached $M$, then a shipment decision is made. In Subsection
\ref{subsec:Martingale-Argument-for IRP}, we already justified that
the instantaneous rate policy is superior to the other renewal type
clearing policies in terms of average cost. However, it could not
guarantee a maximum waiting time for the customers. In contrast, the
instantaneous rate hybrid policy sets a maximum waiting time for the
customers.

In the following, we calculate AWDR for the three classes of clearing
policies, and provide some comparative results in terms of AWDR.

1. IRP with parameter $M$: $\tau=\tau_{M}$, $\sum_{i=1}^{n}\omega_{i}N_{i}(\tau_{M})=M$.
So,
\begin{eqnarray*}
 &  & \mathrm{\mathbb{E}}[W_{IRP}]=\frac{1}{2\sum_{i=1}^{n}\omega_{i}D_{i}}M^{2}-\frac{\sum_{i=1}^{n}\omega_{i}^{2}\sigma_{i}^{2}}{2(\sum_{i=1}^{n}\omega_{i}D_{i})^{2}}M,\\
 &  & \mathrm{\mathbb{E}}[L_{IRP}]=\mathrm{\mathbb{E}}[\tau_{M}]=\frac{M}{\sum_{i=1}^{n}\omega_{i}D_{i}}.
\end{eqnarray*}

2. TP with parameter $T$: $\tau=T$, and
\[
\sum_{i=1}^{n}\omega_{i}N_{i}(T)\sim Normal(\sum_{i=1}^{n}\omega_{i}D_{i}T,\sum_{i=1}^{n}\omega_{i}^{2}\sigma_{i}^{2}T).
\]
So,
\begin{eqnarray*}
\mathrm{\mathbb{E}}[W_{TP}]=\frac{1}{2}\sum_{i=1}^{n}\omega_{i}D_{i}T^{2},\mathrm{\mathbb{E}}[L_{TP}]=T.
\end{eqnarray*}

3. IRHP with parameters $M$ and $T$: $\tau=\tau_{M}\wedge T$.
\begin{eqnarray*}
 &  & \mathrm{\mathbb{E}}[W_{IRHP}]=\frac{1}{2\sum_{i=1}^{n}\omega_{i}D_{i}}\mathrm{\mathbb{E}}\left[(\sum_{i=1}^{n}\omega_{i}N_{i}(\tau_{M}\wedge T))^{2}\right]-\frac{\sum_{i=1}^{n}\omega_{i}^{2}\sigma_{i}^{2}}{2(\sum_{i=1}^{n}\omega_{i}D_{i})^{2}}\mathrm{\mathbb{E}}\left[\sum_{i=1}^{n}\omega_{i}N_{i}(\tau_{M}\wedge T)\right],\\
 &  & \mathrm{\mathbb{E}}[L_{IRHP}]=\mathrm{\mathbb{E}}[\tau_{M}\wedge T].
\end{eqnarray*}

In Table \ref{ExprofAWDR}, we summarize the AWDR for different clearing
policies. The goal of this section is to provide some comparative
results in terms of AWDR, which will be stated in Subsection \ref{subsec:Comparison-of-AWDR}.
In particular, we are interested in comparing IRHP and time policy,
in terms of AWDR. To obtain this comparative result, we need an inequality
which is presented in Subsection \ref{subsec:A-Key-Inequality}.

\begin{table}
\centering \hspace*{-0.5in} %
\begin{tabular}{llll}
\hline
\vspace{-0.1in}
  &  &  & \tabularnewline
$AWDR_{\tau}$  & $=$  & {\Large{}{}$\;\frac{\frac{1}{2\sum_{i=1}^{n}\omega_{i}D_{i}}\mathrm{\mathbb{E}}[(\sum_{i=1}^{n}\omega_{i}N_{i}(\tau))^{2}]-\frac{\sum_{i=1}^{n}\omega_{i}^{2}\sigma_{i}^{2}}{2(\sum_{i=1}^{n}\omega_{i}D_{i})^{2}}\mathrm{\mathbb{E}}[\sum_{i=1}^{n}\omega_{i}N_{i}(\tau)]}{\mathrm{\mathbb{E}}[\tau]}$}  & \vspace{0.05in}
 \tabularnewline
$AWDR_{IRP}$  & $=$  & {\Large{}{}$\;\frac{\frac{1}{2\sum_{i=1}^{n}\omega_{i}D_{i}}M^{2}-\frac{\sum_{i=1}^{n}\omega_{i}^{2}\sigma_{i}^{2}}{2(\sum_{i=1}^{n}\omega_{i}D_{i})^{2}}M}{\frac{M}{\sum_{i=1}^{n}\omega_{i}D_{i}}}=\frac{M-\frac{\sum_{i=1}^{n}\omega_{i}^{2}\sigma_{i}^{2}}{\sum_{i=1}^{n}\omega_{i}D_{i}}}{2}$}  & \vspace{0.05in}
 \tabularnewline
$AWDR_{TP}$  & $=$  & {\Large{}{}$\;\frac{\frac{1}{2}\sum_{i=1}^{n}\omega_{i}D_{i}T^{2}}{T}=\frac{\sum_{i=1}^{n}\omega_{i}D_{i}T}{2}$}  & \vspace{0.05in}
 \tabularnewline
$AWDR_{IRHP}$  & $=$  & {\Large{}{}$\;\frac{\frac{1}{2\sum_{i=1}^{n}\omega_{i}D_{i}}\mathrm{\mathbb{E}}[(\sum_{i=1}^{n}\omega_{i}N_{i}(\tau_{M}\wedge T))^{2}]-\frac{\sum_{i=1}^{n}\omega_{i}^{2}\sigma_{i}^{2}}{2(\sum_{i=1}^{n}\omega_{i}D_{i})^{2}}\mathrm{\mathbb{E}}[\sum_{i=1}^{n}\omega_{i}N_{i}(\tau_{M}\wedge T)]}{\mathrm{\mathbb{E}}[\tau_{M}\wedge T]}$}  & \vspace{0.1in}
 \tabularnewline
\hline
\end{tabular}\caption{Summary of the Expressions of $AWDR$.}
\label{ExprofAWDR}
\end{table}

\subsection{A Key Inequality\label{subsec:A-Key-Inequality}}

Next lemma ables us to establish a comparative result between IRHP
and TP in terms of AWDR. We use the new notation $(x)_{+}=\max(x,0)$.
\begin{lem}
\label{inequ-single} Let $N(t)=Dt+\sigma B(t)$ be a Brownian motion
with drift and denote its hitting times $\tau_{q}=\min\{t:\,N(t)=q\}$
for $q>0$. Fix $T>0$, then
\begin{eqnarray*}
 &  & \frac{\sigma^{2}}{D}\,\mathrm{\mathbb{E}}[N(\tau_{q}\wedge T)]-\mathbb{VAR}[N(\tau_{q}\wedge T)]\\
= &  & D^{2}\left(\mathbb{VAR}[\tau_{q}\wedge T]+2\mathrm{\mathbb{E}}[(\tau_{q}-T)_{+}]\mathrm{\mathbb{E}}[(T-\tau_{q})_{+}]\right)>0.
\end{eqnarray*}
\end{lem}

\begin{proof}
First, since $N(t)-Dt$ and $(N(t)-Dt)^{2}-\sigma^{2}t$ are two martingales,
then we have by the martingale stopping theorem,

\begin{equation}
\mathrm{\mathbb{E}}[N(\tau_{q}\wedge T)]=D\mathrm{\mathbb{E}}[\tau_{q}\wedge T],\label{Ward's first}
\end{equation}
and

\begin{equation}
\mathrm{\mathbb{E}}[(N(\tau_{q}\wedge T)-D(\tau_{q}\wedge T))^{2}]=\sigma^{2}\mathrm{\mathbb{E}}[\tau_{q}\wedge T]\label{Ward's second}
\end{equation}

Using (\ref{Ward's first}) (\ref{Ward's second}), and then simplifying,
\begin{eqnarray}
 &  & \frac{\sigma^{2}}{D}\,\mathrm{\mathbb{E}}[N(\tau_{q}\wedge T)]-\mathbb{VAR}[N(\tau_{q}\wedge T)]\nonumber \\
= &  & \sigma^{2}\,\mathrm{\mathbb{E}}[\tau_{q}\wedge T]-\mathbb{VAR}[N(\tau_{q}\wedge T)]\nonumber \\
= &  & \mathrm{\mathbb{E}}[(N(\tau_{q}\wedge T)-D(\tau_{q}\wedge T))^{2}]-\mathrm{\mathbb{E}}[(N(\tau_{q}\wedge T))^{2}]+D^{2}\mathrm{\mathbb{E}}^{2}[\tau_{q}\wedge T]\nonumber \\
= &  & D^{2}\mathrm{\mathbb{E}}[(\tau_{q}\wedge T)^{2}]+D^{2}\mathrm{\mathbb{E}}^{2}[\tau_{q}\wedge T]-2D\mathrm{\mathbb{E}}[(\tau_{q}\wedge T)N(\tau_{q}\wedge T)].\label{key}
\end{eqnarray}
Next,
\begin{eqnarray}
\mathrm{\mathbb{E}}[\tau_{q}\wedge T]=T-\mathrm{\mathbb{E}}[(T-\tau_{q})1_{\tau_{q}\le T}]=T-\mathrm{\mathbb{E}}[(T-\tau_{q})_{+}].\label{key1}
\end{eqnarray}
Likewise,
\begin{eqnarray}
\mathrm{\mathbb{E}}[(\tau_{q}\wedge T)^{2}] & = & T^{2}-\mathrm{\mathbb{E}}[(T^{2}-\tau_{q}^{2})1_{\tau_{q}\le T}]=T^{2}-\mathrm{\mathbb{E}}[(T+\tau_{q})(T-\tau_{q})_{+}]\nonumber \\
 & = & T^{2}-2T\mathrm{\mathbb{E}}[(T-\tau_{q})_{+}]+\mathrm{\mathbb{E}}[(T-\tau_{q})_{+}^{2}],\label{key2}
\end{eqnarray}
having noted that $T-\tau_{q}=(T-\tau_{q})_{+}-(\tau_{q}-T)_{+}$
and $(T-\tau_{q})_{+}(\tau_{q}-T)_{+}=0$. Applying the strong Markov
property,
\begin{eqnarray}
 & \mathrm{\mathbb{E}}[(\tau_{q}\wedge T)N(\tau_{q}\wedge T)] & =\mathrm{\mathbb{E}}[TN(T)+(q\tau_{q}-TN(T))1_{\tau_{q}\le T}]\nonumber \\
 &  & =DT^{2}+\mathrm{\mathbb{E}}[(q(\tau_{q}-T)-T(N(T)-N(\tau_{q})))1_{\tau_{q}\le T}]\nonumber \\
 &  & =DT^{2}+\mathrm{\mathbb{E}}[(q(\tau_{q}-T)-DT(T-\tau_{q}))1_{\tau_{q}\le T}]\nonumber \\
 &  & =DT^{2}-D\mathrm{\mathbb{E}}[(T+\mathrm{\mathbb{E}}[\tau_{q}])(T-\tau_{q})_{+}]\nonumber \\
 &  & =DT^{2}-2DT\mathrm{\mathbb{E}}[(T-\tau_{q})_{+}]+D\mathrm{\mathbb{E}}[T-\tau_{q}]\mathrm{\mathbb{E}}[(T-\tau_{q})_{+}].\label{key3}
\end{eqnarray}
Putting (\ref{key1}) (\ref{key2})(\ref{key3}) into (\ref{key}),
\begin{eqnarray*}
 &  & \frac{\sigma^{2}}{D}\,\mathrm{\mathbb{E}}[N(\tau_{q}\wedge T)]-\mathbb{VAR}[N(\tau_{q}\wedge T)]\\
= &  & D^{2}(\mathrm{\mathbb{E}}[(T-\tau_{q})_{+}^{2}]+\mathrm{\mathbb{E}}^{2}[(T-\tau_{q})_{+}]-2\mathrm{\mathbb{E}}[T-\tau_{q}]\mathrm{\mathbb{E}}[(T-\tau_{q})_{+}])\\
= &  & D^{2}(\mathrm{\mathbb{E}}[(T-\tau_{q})_{+}^{2}]-\mathrm{\mathbb{E}}^{2}[(T-\tau_{q})_{+}]+2\mathrm{\mathbb{E}}[(\tau_{q}-T)_{+}]\mathrm{\mathbb{E}}[(T-\tau_{q})_{+}])\\
= &  & D^{2}(\mathbb{VAR}[(T-\tau_{q})_{+}]+2\mathrm{\mathbb{E}}[(\tau_{q}-T)_{+}]\mathrm{\mathbb{E}}[(T-\tau_{q})_{+}])\\
= &  & D^{2}(\mathbb{VAR}[\tau_{q}\wedge T]+2\mathrm{\mathbb{E}}[(\tau_{q}-T)_{+}]\mathrm{\mathbb{E}}[(T-\tau_{q})_{+}])>0.
\end{eqnarray*}

Based on Lemma \ref{inequ-single}, we immediately obtain the next
result, which will play an important role in establishing a comparative
result between IRHP and time policy in terms AWDR in Theorem \ref{WHP and TP}.
\end{proof}
\begin{lem}
\label{inequ-multi} Fix $T>0$, then
\begin{eqnarray*}
 &  & \frac{\sum_{i=1}^{n}\omega_{i}^{2}\sigma_{i}^{2}}{\sum_{i=1}^{n}\omega_{i}D_{i}}\mathrm{\mathbb{E}}[\sum_{i=1}^{n}\omega_{i}N_{i}(\tau_{M}\wedge T)]-\mathbb{VAR}[\sum_{i=1}^{n}\omega_{i}N_{i}(\tau_{M}\wedge T)]>0.
\end{eqnarray*}
\end{lem}

\begin{proof}
Treating $\sum_{i=1}^{n}\omega_{i}N_{i}(t)$ as a one dimensional
drifted Brownian motion with drift $\sum_{i=1}^{n}\omega_{i}D_{i}$
and diffusion coefficient $\sqrt{\sum_{i=1}^{n}\omega_{i}^{2}\sigma_{i}^{2}}$,
and applying Lemma \ref{inequ-single}, we arrive at the conclusion.
\end{proof}

\subsection{Comparison of AWDR under the Fixed Clearing Frequency\label{subsec:Comparison-of-AWDR}}

In this subsection, we compare AWDR among different clearing policies
under the fixed clearing frequency. This comparison is proposed because
there is an inherent tradeoff between the clearing frequency and AWDR.
This comparison in shipment consolidation setting with a Poisson process
demand is numerically studied in \cite{CMW14}. First, we demonstrate
the optimality of IRP in terms of AWDR, under a fixed clearing frequency.
\begin{thm}
\label{IRPdominanceAWDR} In terms of AWDR, IRP outperforms all renewal
type clearing policies with which cycle times are of finite second
moment, under a fixed clearing frequency.
\end{thm}

\begin{proof}
From Table \ref{ExprofAWDR}, we know AWDR of a clearing policy with
clearing time $\tau$ is
\[
AWDR_{\tau}=\frac{\frac{1}{2\sum_{i=1}^{n}\omega_{i}D_{i}}\mathrm{\mathbb{E}}[(\sum_{i=1}^{n}\omega_{i}N_{i}(\tau))^{2}]-\frac{\sum_{i=1}^{n}\omega_{i}^{2}\sigma_{i}^{2}}{2(\sum_{i=1}^{n}\omega_{i}D_{i})^{2}}\mathrm{\mathbb{E}}[\sum_{i=1}^{n}\omega_{i}N_{i}(\tau)]}{\mathrm{\mathbb{E}}[\tau]}.
\]
Noticing the fixed $\mathrm{\mathbb{E}}[\tau]$ implies $\mathrm{\mathbb{E}}[\sum_{i=1}^{n}\omega_{i}N_{i}(\tau)]$
is fixed, we have
\begin{eqnarray*}
AWDR_{\tau}\geq\frac{\frac{1}{2\sum_{i=1}^{n}\omega_{i}D_{i}}\mathrm{\mathbb{E}}^{2}[\sum_{i=1}^{n}\omega_{i}N_{i}(\tau)]-\frac{\sum_{i=1}^{n}\omega_{i}^{2}\sigma_{i}^{2}}{2(\sum_{i=1}^{n}\omega_{i}D_{i})^{2}}\mathrm{\mathbb{E}}[\sum_{i=1}^{n}\omega_{i}N_{i}(\tau)]}{\mathrm{\mathbb{E}}[\tau]},
\end{eqnarray*}
the equality holds if and only if $\mathrm{\mathbb{E}}[\sum_{i=1}^{n}\omega_{i}N_{i}(\tau)]$
is a constant, which implies IRP achieves the least AWDR with a fixed
clearing frequency.
\end{proof}
One disadvantage of IRP is that it has no upper bound on the cycle
time, in contrast, IRHP are of practical importance since by definition
it has an upper bound on the cycle time. This observation enhances
the value of next result, in which we are able to compare IRHP with
TP, justifying the advantage of IRHP.
\begin{thm}
\label{WHP and TP} Under a fixed clearing frequency, IRHP performs
better than TP, in terms of AWDR.
\end{thm}

\begin{proof}
We consider a fixed $\mathrm{\mathbb{E}}[\tau]$ and use the following
notation for the corresponding policy parameters under this $\mathrm{\mathbb{E}}[\tau]$
value: TP with parameter $T$, and IRHP with parameters $M_{H}$ and
$T_{H}$. Recalling the $\mathrm{\mathbb{E}}[\tau]$ expressions for
different policies in Table \ref{ExprofAWDR}, we note that, by assumption,
\begin{eqnarray*}
\mathrm{\mathbb{E}}[\tau_{M_{H}}\wedge T_{H}]=T,
\end{eqnarray*}
which implies
\begin{eqnarray}
\mathrm{\mathbb{E}}[\sum_{i=1}^{n}\omega_{i}N_{i}(\tau_{M_{H}}\wedge T_{H})]=\sum_{i=1}^{n}\omega_{i}D_{i}T.\label{WHP-TP}
\end{eqnarray}
Next, recalling the results in Table \ref{ExprofAWDR} and the assumption
of fixed $\mathrm{\mathbb{E}}[\tau]$ values for all the policies
of interest, we need to show that

\begin{equation}
\mathrm{\mathbb{E}}\left[(\sum_{i=1}^{n}\omega_{i}N_{i}(\tau_{M_{H}}\wedge T_{H}))^{2}\right]-\frac{\sum_{i=1}^{n}\omega_{i}^{2}\sigma_{i}^{2}}{\sum_{i=1}^{n}\omega_{i}D_{i}}\mathrm{\mathbb{E}}\left[\sum_{i=1}^{n}\omega_{i}N_{i}(\tau_{M_{H}}\wedge T_{H})\right]<(\sum_{i=1}^{n}\omega_{i}D_{i})^{2}T^{2}.\label{WHP vs TP}
\end{equation}

In fact, by recalling (\ref{WHP-TP}) and Lemma \ref{inequ-multi},
(\ref{WHP vs TP}) is verified.
\end{proof}
\begin{rem}
From Lemma \ref{WHP and TP}, we can conclude that, given any time
policy and IRHP, as long as they have the same clearing frequency,
the IRHP achieves less AWDR and average cost than the TP. This argument
justifies the advantage of IRHP, which inherits the merits of both
IRP and TP.
\end{rem}

\section{Conclusions\label{sec:Conclusions}}

In this work, we first consider the average cost model of a stochastic
clearing system with multiple drifted Brownian motion inputs. In the
single item case, the quantity-based policy is always the best one.
However, we show that in multi-item case, this result does not hold.
We further identify a set of $(T_{Q}+T)$ type policies, and obtain
a somewhat surprising result that the jointly optimal $(T_{Q}+T)$-policy
is either the optimal quantity based policy or the optimal time based
policy. Later on, we propose an instantaneous rate policy (IRP) and
show that the optimal instantaneous rate policy achieves the least
average cost among a large class of renewal type clearing policies
by applying a martingale-based argument. In this stochastic clearing
model with multiple input processes, time-based policy is an open-loop
policy, without need to track any process realization; quantity-based
policy can be considered as a semi closed-loop policy, which only
requires to track the sum of all input processes as a whole; the instantaneous
rate policy is a truly closed-loop policy since it requires to track
realizations of all input processes. From the perspective of information
value, the optimal policy in stochastic dynamic systems should always
be a closed-loop policy.

Second, we consider stochastic clearing systems from the service performance
perspective. In particular, we propose measuring performance with
the average weighted delay rate. We show that for a given expected
clearing cycle length, IRP outperforms a large class of policies in
terms of AWDR. Obviously, one disadvantage of IRP is that it has no
upper bound on the cycle time. IRHP are of practical importance since
by definition it has an upper bound on the cycle time. More interestingly,
we show that for a fixed clearing frequency, the IRHP performs better
than TP in terms of AWDR (also in terms of average cost), which justifies
the advantage of IRHP.

Important extensions of stochastic clearing systems with multiple
input processes studied here include the integrated inventory/consolidation
problem (see \cite{CL00,CML06,CTL08}), and dynamic pricing problem
which is a revenue maximization model through price-based control
towards rates of input processes.

 \bibliographystyle{plain}
\bibliography{references}

\appendix

\section{Appendix}
\begin{proof}[Proof of Lemma \ref{joint-quantity}]
We show this result with $i=1$. For $i=2,3,\ldots,n$, we can show
it with the same method. From $\sum_{i=1}^{n}D_{i}T_{Q}+\sum_{i=1}^{n}\sigma_{i}B_{i}(T_{Q})=Q$,
we have
\[
B_{n}(T_{Q})=\frac{Q-\sum_{i=1}^{n}D_{i}T_{Q}-\sum_{i=1}^{n-1}\sigma_{i}B_{i}(T_{Q})}{\sigma_{n}},
\]
then for any real numbers $a_{i}$, $i=1,2,\ldots,n$, we have
\begin{eqnarray*}
 &  & \sum_{i=1}^{n}a_{i}B_{i}(T_{Q})-\frac{1}{2}\sum_{i=1}^{n}a_{i}^{2}T_{Q}\\
= &  & \sum_{i=1}^{n-1}a_{i}B_{i}(T_{Q})+a_{n}\frac{Q-\sum_{i=1}^{n}D_{i}T_{Q}-\sum_{i=1}^{n-1}\sigma_{i}B_{i}(T_{Q})}{\sigma_{n}}-\frac{1}{2}\sum_{i=1}^{n}a_{i}^{2}T_{Q}\\
= &  & \sum_{i=1}^{n-1}(a_{i}-\frac{a_{n}}{\sigma_{n}}\sigma_{i})B_{i}(T_{Q})-\sum_{i=1}^{n}(\frac{a_{n}}{\sigma_{n}}D_{i}+\frac{1}{2}a_{i}^{2})T_{Q}+\frac{a_{n}}{\sigma_{n}}Q.
\end{eqnarray*}
By equaling $\sum_{i=1}^{n-1}(a_{i}-\frac{a_{n}}{\sigma_{n}}\sigma_{i})B_{i}(T_{Q})-\sum_{i=1}^{n}(\frac{a_{n}}{\sigma_{n}}D_{i}+\frac{1}{2}a_{i}^{2})T_{Q}=s_{1}B_{1}(T_{Q})+s_{2}T_{Q}$,
we obtain
\begin{eqnarray}
\left\{ \begin{array}{l}
a_{1}-\frac{a_{n}}{\sigma_{n}}\sigma_{1}=s_{1}\\
a_{2}-\frac{a_{n}}{\sigma_{n}}\sigma_{2}=0\\
\vdots\\
a_{n-1}-\frac{a_{n}}{\sigma_{n}}\sigma_{n-1}=0\\
-\sum_{i=1}^{n}(\frac{a_{n}}{\sigma_{n}}D_{i}+\frac{1}{2}a_{i}^{2})=s_{2}
\end{array}\right.\Longrightarrow\left\{ \begin{array}{l}
a_{1}=\frac{a_{n}}{\sigma_{n}}\sigma_{1}+s_{1}\\
a_{2}=\frac{a_{n}}{\sigma_{n}}\sigma_{2}\\
\vdots\\
a_{n-1}=\frac{a_{n}}{\sigma_{n}}\sigma_{n-1}\\
\sum_{i=1}^{n}\frac{1}{2}a_{i}^{2}+\frac{a_{n}}{\sigma_{n}}\sum_{i=1}^{n}D_{i}+s_{2}=0
\end{array}\right.\label{a_i}
\end{eqnarray}
and arrive at
\begin{eqnarray*}
 &  & \frac{1}{2}(\frac{a_{n}}{\sigma_{n}}\sigma_{1}+s_{1})^{2}+\frac{1}{2}(\frac{a_{n}}{\sigma_{n}})^{2}\sum_{i=2}^{n}\sigma_{i}^{2}+\frac{a_{n}}{\sigma_{n}}D+s_{2}\\
= &  & \frac{1}{2}\sigma^{2}(\frac{a_{n}}{\sigma_{n}})^{2}+(s_{1}\sigma_{1}+D)\frac{a_{n}}{\sigma_{n}}+(\frac{1}{2}s_{1}^{2}+s_{2})=0
\end{eqnarray*}
We take the positive root

\begin{equation}
\frac{a_{n}}{\sigma_{n}}=\frac{-(s_{1}\sigma_{1}+D)+\sqrt{(s_{1}\sigma_{1}+D)^{2}-(s_{1}^{2}+2s_{2})\sigma^{2}}}{\sigma^{2}}>0.\label{an/sigman}
\end{equation}
Since $s_{1}^{2}+2s_{2}<0$, there exist $\epsilon>0,\delta>0$ such
that

\begin{equation}
(1+\epsilon)(1+\delta)s_{1}^{2}+2s_{2}=0.\label{s1s2delta}
\end{equation}
In the following, we want to show that $\left\{ \exp(\sum_{i=1}^{n}a_{i}B_{i}(T_{Q}\wedge t)-\frac{1}{2}\sum_{i=1}^{n}a_{i}^{2}(T_{Q}\wedge t))\right\} _{t\geq0}$
is a uniformly integrable martingale. According to \cite[Proposition 2.5.7(ii)]{AL06},
it's sufficient to show that for all $t\geq0$,

\[
\mathrm{\mathbb{E}}\left[\left(\exp(\sum_{i=1}^{n}a_{i}B_{i}(T_{Q}\wedge t)-\sum_{i=1}^{n}\frac{1}{2}a_{i}^{2}(T_{Q}\wedge t))\right)^{1+\delta}\right]<\infty.
\]
For any fixed $t\geq0$, using (\ref{a_i}),
\begin{eqnarray}
 &  & \mathrm{\mathbb{E}}\left[\left(\exp(\sum_{i=1}^{n}a_{i}B_{i}(T_{Q}\wedge t)-\sum_{i=1}^{n}\frac{1}{2}a_{i}^{2}(T_{Q}\wedge t))\right)^{1+\delta}\right]\nonumber \\
= &  & \mathrm{\mathbb{E}}\left[\exp\left((1+\delta)\sum_{i=1}^{n}a_{i}B_{i}(T_{Q}\wedge t)-(1+\delta)\sum_{i=1}^{n}\frac{1}{2}a_{i}^{2}(T_{Q}\wedge t)\right)\right]\nonumber \\
= &  & \mathrm{\mathbb{E}}\left[\exp\left((1+\delta)\frac{a_{n}}{\sigma_{n}}\left(D(T_{Q}\wedge t)+\sum_{i=1}^{n}\sigma_{i}B_{i}(T_{Q}\wedge t)\right)+(1+\delta)\left(s_{1}B_{1}(T_{Q}\wedge t)+s_{2}(T_{Q}\wedge t)\right)\right)\right].\label{UIcheck}
\end{eqnarray}
By Hölder's inequality \cite[Theorem 3.1.11]{AL06}, we obtain that

\begin{eqnarray}
 &  & \mathrm{\mathbb{E}}\left[\exp\left((1+\delta)\frac{a_{n}}{\sigma_{n}}\left(D(T_{Q}\wedge t)+\sum_{i=1}^{n}\sigma_{i}B_{i}(T_{Q}\wedge t)\right)+(1+\delta)\left(s_{1}B_{1}(T_{Q}\wedge t)+s_{2}(T_{Q}\wedge t)\right)\right)\right]\nonumber \\
\leq &  & \mathrm{\mathbb{E}}^{\frac{\epsilon}{1+\epsilon}}\left[\exp\left(\frac{(1+\epsilon)(1+\delta)}{\epsilon}\frac{a_{n}}{\sigma_{n}}\left(D(T_{Q}\wedge t)+\sum_{i=1}^{n}\sigma_{i}B_{i}(T_{Q}\wedge t)\right)\right)\right]\nonumber \\
 &  & \times\mathrm{\mathbb{E}}^{\frac{1}{1+\epsilon}}\left[\exp\left((1+\epsilon)(1+\delta)(s_{1}B_{1}(T_{Q}\wedge t)+s_{2}(T_{Q}\wedge t))\right)\right].\label{Holder inequality}
\end{eqnarray}
From $\frac{a_{n}}{\sigma_{n}}>0$ and $D(T_{Q}\wedge t)+\sum_{i=1}^{n}\sigma_{i}B_{i}(T_{Q}\wedge t)\leq Q$,
for all $t\geq0$, we have
\begin{equation}
\mathrm{\mathbb{E}}^{\frac{\epsilon}{1+\epsilon}}\left[\exp\left(\frac{(1+\epsilon)(1+\delta)}{\epsilon}\frac{a_{n}}{\sigma_{n}}\left(D(T_{Q}\wedge t)+\sum_{i=1}^{n}\sigma_{i}B_{i}(T_{Q}\wedge t)\right)\right)\right]\leq\exp\left((1+\delta)\frac{a_{n}}{\sigma_{n}}Q\right).\label{lessthanQ}
\end{equation}
Recalling (\ref{s1s2delta}) and the optional stopping theorem for
exponential martingale, we have for all $t\geq0$,

\begin{eqnarray}
= &  & \mathrm{\mathbb{E}}\left[\exp\left((1+\epsilon)(1+\delta)(s_{1}B_{1}(T_{Q}\wedge t)+s_{2}(T_{Q}\wedge t))\right)\right]\nonumber \\
= &  & \mathrm{\mathbb{E}}\left[\exp\left((1+\epsilon)(1+\delta)s_{1}B_{1}(T_{Q}\wedge t)-\frac{1}{2}(1+\epsilon)^{2}(1+\delta)^{2}s_{1}^{2}(T_{Q}\wedge t)\right)\right]\nonumber \\
= &  & 1.\label{exponentialmartingale}
\end{eqnarray}
From (\ref{UIcheck}), (\ref{Holder inequality}), (\ref{lessthanQ}),
and (\ref{exponentialmartingale}), we have $\mathrm{\mathbb{E}}\left[\left(\exp(\sum_{i=1}^{n}a_{i}B_{i}(T_{Q}\wedge t)-\sum_{i=1}^{n}\frac{1}{2}a_{i}^{2}(T_{Q}\wedge t))\right)^{1+\delta}\right]<\infty$
for all $t\geq0$, which implies that $\left\{ \exp(\sum_{i=1}^{n}a_{i}B_{i}(T_{Q}\wedge t)-\frac{1}{2}\sum_{i=1}^{n}a_{i}^{2}(T_{Q}\wedge t))\right\} _{t\geq0}$
is a uniformly integrable martingale. Therefore, by the optional stopping
theorem and Vitali convergence theorem,
\[
\mathrm{\mathbb{E}}\left[\exp(\sum_{i=1}^{n}a_{i}B_{i}(T_{Q})-\frac{1}{2}\sum_{i=1}^{n}a_{i}^{2}T_{Q})\right]=1.
\]
Recalling that $(a_{1},a_{2},\ldots,a_{n})$ in (\ref{a_i}) are selected
such that $\sum_{i=1}^{n}a_{i}B_{i}(T_{Q})-\frac{1}{2}\sum_{i=1}^{n}a_{i}^{2}T_{Q}=s_{1}B_{1}(T_{Q})+s_{2}T_{Q}+\frac{a_{n}}{\sigma_{n}}Q$,
and (\ref{an/sigman}), we have
\begin{eqnarray*}
 &  & \mathrm{\mathbb{E}}\left[\exp(s_{1}B_{1}(T_{Q})+s_{2}T_{Q})\right]\\
= &  & \exp\left(\frac{s_{1}\sigma_{1}+D-\sqrt{(s_{1}\sigma_{1}+D)^{2}-(s_{1}^{2}+2s_{2})\sigma^{2}}}{\sigma^{2}}Q\right),
\end{eqnarray*}
which is the joint moment generation function for $(B_{1}(T_{Q}),T_{Q})$.
Moreover,
\[
\mathrm{\mathbb{E}}[B_{1}(T_{Q})T_{Q}]=\frac{\partial^{2}\mathrm{\mathbb{E}}[\exp(s_{1}B_{1}(T_{Q})+s_{2}T_{Q})]}{\partial s_{1}\partial s_{2}}|_{s_{1}=s_{2}=0}=-\frac{\sigma_{1}Q}{D^{2}}.
\]
\end{proof}

\end{document}